\documentclass[11pt,a4paper]{article}

\usepackage{amsfonts,amsmath,amsthm,amssymb,amstext}
\usepackage{geometry}                
\usepackage{fullpage}
\usepackage{graphicx}
\usepackage{subfigure}
\usepackage{latexsym}

\usepackage[pdftex,bookmarks,colorlinks,breaklinks]{hyperref}  
\usepackage{color}

\definecolor{dullmagenta}{rgb}{0.4,0,0.4}   
\definecolor{darkblue}{rgb}{0,0,0.4}
\hypersetup{linkcolor=red,citecolor=blue,filecolor=dullmagenta,urlcolor=darkblue} 

\newtheorem{thm}{Theorem}[section]
\newtheorem{cor}[thm]{Corollary}

\newtheorem{lem}[thm]{Lemma}
\newtheorem{rem}[thm]{Remark}

\usepackage{algorithm}
\usepackage{algorithmic}

\title{The local convexity of solving systems of quadratic equations}

\date{\today} 

\author{Chris D. White\thanks{Department of Mathematics, University of Texas at Austin, Austin, TX 78712 USA ({\tt cwhite@math.utexas.edu})}, Sujay Sanghavi\thanks{Department of Electrical and Computer Engineering, University of Texas at Austin, Austin, TX 78712 USA ({\tt sanghavi@mail.utexas.edu})}, and  Rachel Ward\thanks{Department of Mathematics, University of Texas at Austin, Austin, TX 78712 USA ({\tt rward@math.utexas.edu})}}

\begin{document}
\maketitle
\begin{abstract} 
This paper considers the recovery of a rank $r$ positive semidefinite matrix $X X^T\in\mathbb{R}^{n\times n}$ from $m$ scalar measurements of the form $y_i := a_i^T X X^T a_i$ (i.e., quadratic measurements of $X$).  Such problems arise in a variety of applications, including covariance sketching of high-dimensional data streams, quadratic regression, quantum state tomography, among others.  A natural approach to this problem is to minimize the loss function $f(U) =  \sum_i (y_i - a_i^TUU^Ta_i)^2$ which has an entire manifold of solutions given by $\{XO\}_{O\in\mathcal{O}_r}$ where $\mathcal{O}_r$ is the orthogonal group of $r\times r$ orthogonal matrices; this is {\it non-convex} in the $n\times r$ matrix $U$, but methods like gradient descent are simple and easy to implement (as compared to semidefinite relaxation approaches). 

In this paper we show that once we have $m \geq C nr \log^2(n)$ samples from isotropic gaussian $a_i$, with high probability {\em (a)} this function admits a dimension-independent region of {\em local strong convexity} on lines perpendicular to the solution manifold, and {\em (b)} with an additional polynomial factor of $r$ samples, a simple spectral initialization will land within the region of convexity with high probability. Together, this implies that gradient descent with initialization (but no re-sampling) will converge linearly to the correct $X$, up to an orthogonal transformation.  We believe that this general technique (local convexity reachable by spectral initialization) should prove applicable to a broader class of nonconvex optimization problems.
\end{abstract}

\section{Introduction}\label{sec:intro}

Consider $X\in \mathbb{R}^{n\times r}$ fixed and unknown, acquired through quadratic measurements of the form $y_i:=a_i^T X X^T a_i = \text{tr}(XX^Ta_ia_i^T) = \|X^Ta_i\|_2^2$, $i = 1 \dots m$.   Assuming $X$ has full column rank, this is equivalent to receiving noiseless samples $a_i^TMa_i$ of the positive semidefinite matrix $M:=XX^T$.  Scenarios such as this arise in various applications: for a concrete example, suppose we receive a stream of high-dimensional centered Gaussian vectors $\{x_j\}_{j=1}^m$ with unknown covariance matrix $\Sigma$.  If we believe $\Sigma$ to be (approximately) low rank, we can recover this structure via the sampled matrix
$$
\hat{\Sigma}=\frac{1}{m}\sum_j x_jx_j^T.
$$
However, due to the large dimensionality, storing all of the incoming vectors $x_i$ might be prohibitive.  Instead, we randomly draw a set of sensing vectors $\{a_k\}$ which are efficient to store (e.g., they are \emph{sparse}) and for each incoming data point compute $y_{ki} = (a_k^Tx_i)^2$.  We are now only storing $(y_{ki}, a_k)_{k,i}$ which is a sparse data set.  Note that if we define
$$
y_k := \frac{1}{m}\sum_i y_{ki}
$$
then $y_k$ has the form 
$$
a_k^T(\frac{1}{m}\sum_j x_jx_j^T)a_k = a_k^T\hat{\Sigma}a_k.
$$
The question posed above is: can we compute the covariance structure of the $\{x_i\}$ given only this data?

The example above describes \emph{covariance sketching} of high-dimensional data streams \cite{covsketch2012,chen2013exact}, but there are many other scenarios that fall under our problem setting, e.g., phaseless measurements in physics and optics \cite{physics1,physics2,phase1,phase2}.   Because this data is invariant under the transformation
$$
X\mapsto XO
$$
for any orthogonal matrix $O\in\mathbb{R}^{r\times r}$, we can only hope to recover $X$ up to this action.  In the complex rank one setting $x\in\mathbb{C}^n$, this means we can only recover $x$ up to a global phase. \emph{Phase retrieval} problems of this type often arise in the physical sciences due to the nature of optical sensors, which can only record intensity information \cite{phase1, phase2}.   It is now well-understood that if the measurement vectors $a_i$ are generic or e.g. independent Gaussian random vectors, then $m \geq O(n)$ measurements suffice for injectivity of the map $x \mapsto \left( |\langle a_i, x \rangle|^2 \right)_{i=1}^m$ up to phase \cite{balan2006, EM14}.  There is still a question of how to perform the inverse map in an efficient and stable manner, and in recent years several different algorithms have been proposed in this direction, see for example \cite{balan2009painless, balan2012reconstruction, sujay2013, candes2013phaselift, candes2014solving, alexeev2014phase, demanet2014stable, EM14, candes2015phase}. 
In particular, \cite{balan2009painless} noted that such measurements may be reformulated as $y_i = \text{tr}\left(a_ia_i^{*} xx^{*} \right),$ so that one can consider this problem as that of recovering an unknown rank-one positive semi-definite matrix.  Inspired by the field of compressive sensing \cite{candes2006near,donoho2006compressed} and low-rank matrix recovery \cite{recht2010guaranteed}, this led to many results demonstrating that well chosen convex and semidefinite programming (SDP) relaxations can provably recover the underlying signal up to phase with only $m\geq Cn$ Gaussian measurements \cite{chai2011array, candes2013phaselift, waldspurger2015phase}.  However, as such algorithms optimize over the ``lifted" space of $n \times n$ positive semidefinite matrices, the computational complexity becomes  quite high.  In the more general rank-$r$ setting, whereby the measurements are $y_i:= \text{tr}(XX^Ta_ia_i^T) = \|X^Ta_i\|_2^2,$ the recent works \cite{kueng2014low,chen2013exact} demonstrate that convex relaxation techniques based on nuclear norm minimization can solve such problems from an optimal number of measurements $O(nr)$, but still require large computational cost.  

In the rank-1 setting in particular, several alternative reconstruction algorithms have been proposed with global phase recovery guarantees which operate directly on the lower-dimensional problem, and thus are more computationally efficient.  Notably, \cite{sujay2013} considers the nonconvex optimization problem
\begin{equation}\label{eq:random_obj1}
\min_{u\in\mathbb{C}^n }\frac{1}{4m}\sum_{i=1}^m(y_i - |a_i^*u|^2)^2
\end{equation}
and proves that after a judiciously chosen initialization, with high probability alternating minimization will converge to the underlying vector $x$ up to phase, assuming random Gaussian measurements. Subsequently \cite{candes2014phase} used the same initialization to show convergence when followed by gradient descent without requiring resampling. Both of these algorithms provably recover the underlying vector $x$ up to global phase, from a number of measurements $m$ which is optimal up to additional logarithmic factors in $n$.  Very recently, the paper \cite{candesnew15} provides a modified gradient method which removes the additional logarithmic factors of $n$ in the number of measurements. 

In a similar vein, many recent works have demonstrated global convergence guarantees for gradient descent on other nonconvex matrix factorization problems.  Specifically, in \cite{zhang2015global} the authors consider gradient descent on the Grassmannian and prove global convergence for a class of SVD problems.  In \cite{de2014global} a stochastic gradient algorithm was shown to converge globally for a low-rank matrix least squares problem.  In \cite{sun2015complete} the authors consider the recovery of a full-rank matrix from sparse linear measurements via manifold optimization over the sphere.  In all of these works including ours, the underlying idea is that the lack of convexity can be fixed by operating on an appropriate matrix manifold.

In this paper, we consider the more general version of problem \eqref{eq:random_obj1} in which the underlying matrix is of rank $r$:
\begin{equation}\label{eq:random_obj}
\min_{U\in\mathbb{R}^{n\times r} }\frac{1}{4m}\sum_{i=1}^m(y_i - \|a_i^TU\|_2^2)^2
\end{equation}
and our measurements take the form
$$
y_i:= \text{tr}(XX^Ta_ia_i^T) = \|X^Ta_i\|_2^2.
$$
As noted in \cite{candes2013phaselift}, it seems unlikely that a deterministic RIP condition holds in this setting.  In any case, our local convexity results are novel and might shed light on other nonconvex problems unrelated to matrix recovery.

Note that throughout the paper we assume that $X$ has full rank, and without loss of generality we can assume its columns are orthogonal.  In contrast to previous algorithms operating directly on the rank one problem \eqref{eq:random_obj1}, we demonstrate that, under a Gaussian assumption on the random measurement vectors, the function \eqref{eq:random_obj} is strongly convex in certain directions, and we can recover $X$ (up to an orthogonal matrix) via spectral initialization followed by gradient descent.  In particular, we prove gradient descent will converge linearly at a rate which depends on the condition number $\lambda_r/\lambda_1$ of the full matrix $XX^T$ and the ambient dimension $n$.  Moreover, the classical phase retrieval problem $x\in\mathbb{C}^n$ can be recast as a rank 2 recovery problem within our framework (see \S\ref{sec:complex}).  Precisely, we show the following all hold with high probability:
\begin{itemize}
\item for general $r$, we demonstrate that after $m\geq Cnr(\log n)^2$ Gaussian samples, in a quantifiable region the function \eqref{eq:random_obj} is \emph{strongly} convex in directions perpendicular to the manifold of solutions

\item The size of this region is \emph{independent} of both the ambient dimension and the rank

\item with an additional factor of $r^5$ samples, a simple spectral initialization will land within this region with high probability and thus standard gradient descent on \eqref{eq:random_obj} will linearly converge to a global minimizer

\item if $r=1$ and $x\in\mathbb{R}^n$ is ``not too peaky" , then for more general sub-gaussian measurements, after $m\geq Cn(\log n)^3$ subgaussian samples, the function appearing in \eqref{eq:random_obj} is \emph{strongly} convex in a ball centered at $x$ (and $-x$), whose radius we explicitly compute.

\end{itemize}

In the real-valued rank one setting, the strong convexity result we present actually holds in much more generality than the initialization result -- for  sub-gaussian measurements -- and we believe this should be of independent interest; in particular, our results hold for Bernoulli measurements and Sparse Gaussian measurements.  We note that in the rank-1 setting, recovery results from general sub-gaussian measurements were also provided in \cite{krahmerphase} using convex optimization for reconstruction, and a similar incoherence condition on the underlying $x$ was also required there.

While preparing this manuscript, we became aware of [\cite{mahdithesis}, p.250] which also certifies local convexity of the function \eqref{eq:random_obj}, for the special case of Gaussian measurements in the rank one setting.

Our results can be viewed as exact recovery guarantees for a special case of a \emph{manifold-constrained least squares problem} where the manifold is the set of rank $r$ positive semidefinite matrices.  This algorithm was studied empirically in \cite{FM15}.  Many nonconvex problems of interest can be reformulated as a manifold-constrained least squares problem, and we believe that the exact recovery guarantees presented here should be extendable to a  broader class of problems.  

\section{Main results}
We aim to solve the nonconvex inverse problem of recovering the unknown matrix $X\in \mathbb{R}^{n\times r}$ (up to right multiplication by an orthogonal matrix) from quadratic measurements of the form
\begin{equation}\label{eq:sample_model}
y_i := \|a_i^TX\|_2^2
\end{equation}
by solving the nonconvex optimization problem
\begin{equation}\label{eq:opt_model}
\min_{U\in\mathbb{R}^{n\times r} } f(U):=\min_{U\in\mathbb{R}^{n\times r} }\frac{1}{4m}\sum_{i=1}^m(y_i - \|a_i^TU\|_2^2)^2.
\end{equation}
Because the function appearing in \eqref{eq:opt_model} is invariant under right multiplication by an orthogonal matrix, there is an entire manifold of solutions given by $\left\{XO : O\in\mathcal{O}(r)\right\}$ where $\mathcal{O}(r)$ is the set of $r\times r$ orthogonal matrices.  Our strategy is to establish that a spectral initialization will land (with high probability) in a region of strong convexity\footnote{Strong convexity here and throughout always refers to convexity in directions orthogonal to the manifold of solutions.}  around the manifold of global minimizers.   An overview of our approach is given in Algorithm \ref{algo:spectral_init}.

\begin{algorithm}[t]
\caption{\label{algo:spectral_init} Initialize and Descend for finding global solutions to \eqref{eq:opt_model}.}
\vspace{.2cm}
\begin{algorithmic}
\STATE{\bfseries Input:} Measurements $y_i = \|a_i^T X \|^2, \quad i = 1,2,\dots, m,$ where $a_i$ are Gaussian\\

\vspace{.2cm}

\STATE {\bf Initialize:} $U_0 = Z\Lambda^{1/2}$, where the columns of $Z$ contain the ($\ell_2$-normalized) eigenvectors corresponding to the $r$ largest eigenvalues $\sigma_1\geq ... \geq \sigma_r$ of the matrix
$$
M:=\frac{1}{2m}\sum_{i=1}^m y_ia_ia_i^T
$$
and the scaling is given by the diagonal matrix
$$\Lambda_i:= \sigma_i-\sigma_{r+1}.$$

\vspace{.2cm}
 
 \STATE{\bf Descend:} Starting at $U_0$, iteratively update $U$ via gradient descent on $f(U)$.

\vspace{.2cm}

 \STATE {\bf Output:} Estimated global solution $\widehat{X}$ to the nonconvex problem \eqref{eq:function}.
 
\vspace{.2cm}
 
 \end{algorithmic}
\end{algorithm}

There are two main ingredients to proving performance guarantees for Algorithm \ref{algo:spectral_init}, namely, the strong convexity of the function $f$ in a region around the manifold of global minimizers at finite sample complexity, and a guarantee that spectral initialization will land within this region.  The finite sample convexity result holds in more generality when $r=1$, while for general $r$ we always assume Gaussian measurements.

\subsection{Rank-$r$ Matrix Recovery}
We focus on the setting where $X\in\mathbb{R}^{n\times r}$ is a fixed unknown matrix with orthogonal columns, and we receive noiseless samples of the form 
\begin{equation}\label{eq:sample_model}
y_i:=\|a_i^TX\|_2^2
\end{equation} for i.i.d. standard Gaussian vectors $\{a_i\}_{i=1}^m$.  Now that we have an entire manifold of solutions given by $\left\{XO : O\in\mathcal{O}(r)\right\}$ we will need to consider the quantity
\begin{equation}\label{eq:dist_to_soln}
d(U):=\min_{O\in\mathcal{O}(r)}\|XO-U\|_F^2
\end{equation}
which is well-defined by compactness of the orthogonal group.  We note that the minimizer may not be unique, but this is not important for our purposes.  We will also need to consider
\begin{equation}\label{eq:eigs}
\lambda_1\geq \lambda_2 ... \geq \lambda_r > 0 
\end{equation}
the non-zero eigenvalues of the positive semidefinite matrix $XX^T$.  For a given matrix $U\in\mathbb{R}^{n\times r}$ the notation $\nabla^2f(U)$ stands for the Hessian of the function \eqref{eq:opt_model}, which is a random $nr\times nr$ symmetric matrix.

The main finite sample convexity result is as follows:
\begin{thm}[Strong Convexity]\label{thm:main_convexity} Suppose we take $m\geq C(\lambda_r/\lambda_1)^{-2}nr(\log n)^2$ samples of the form \eqref{eq:sample_model}, where $\lambda_1$ and $\lambda_r$ are as in \eqref{eq:eigs}; assume further that $U\in\mathbb{R}^{n\times r}$ satisfies
\begin{equation}\label{eq:closeness}
\min_{O\in\mathcal{O}(r)}\|XO-U\|_F < \frac{3\lambda_r}{10\|X\|_F}.
\end{equation}
Then with probability at least $1-4e^{-rn}-7/m^2$, it holds that
\begin{equation*}
\begin{split}
\text{\emph{vec}}(U-XO^*)^T\nabla^2f(U)\text{\emph{vec}}(U-XO^*) &\geq\frac{\lambda_r^2}{18\| X \|_F^2} \| U - X O^{*} \|_F^2 \\
\text{\emph{vec}}(U-XO^*)^T\nabla^2f(U)\text{\emph{vec}}(U-XO^*) &\leq C \left( \frac{n^2r^2\lambda_r^2}{\| X \|_F^2} + \lambda_r  \right)  \| U - X O^{*} \|_F^2
\end{split}
\end{equation*}
where $O^*$ is a minimizer for \eqref{eq:closeness}.  
\end{thm}

For the proof of this theorem, we refer the reader to Section \S\ref{sec:rankr_convexity} below.  Note that this theorem implies that for matrices $U\in\mathbb{R}^{n\times r}$ close to the manifold of solutions, we can control the eigenvalues of the Hessian; in particular for such $U$, \emph{the function $f(U)$ is strongly and uniformly convex on the line connecting $U$ to its nearest point on the manifold of solutions}, as measured by the function \eqref{eq:dist_to_soln}.

We now show how this local strong convexity results in linear convergence to the true $X$ we seek to recover. We have the following theorem which concisely establishes the initialization and performance guarantees of Algorithm \ref{algo:spectral_init}.

\begin{thm}[Main Theorem]\label{thm:main_theorem} Suppose we take $m\geq C\|X\|_F^8\lambda_r^{-4}nr^2(\log n)^2$ samples of the form \eqref{eq:sample_model}, where $\lambda_1$ and $\lambda_r$ are as in \eqref{eq:eigs}. Define the matrix 
$$
M:=\frac{1}{2m}\sum_{i=1}^my_ia_ia_i^T
$$
and the following associated quantities:
\begin{subequations}
\begin{align*}
U&:=\begin{bmatrix}u_1 & u_2 & ... & u_r\end{bmatrix}_{n\times r}\\
\Sigma &:= \begin{bmatrix}\sigma_1 & 0 & ... & 0\\ 0 & \sigma_2 & ... & 0\\ ... & ... & 0 & \sigma_r \end{bmatrix}_{r\times r}-\sigma_{r+1}Id\\
U_0&:=U\Sigma^{1/2}
\end{align*}
\end{subequations}
where $\sigma_1\geq \sigma_2 ... \geq \sigma_{r+1}>0$ are the eigenvalues of $M$ and $u_i$ are the corresponding normalized eigenvectors.
If we iteratively update $U_k$ via gradient descent
$$
U_{k+1} = U_k -\gamma\nabla f(U),
$$
then with probability at least $1-3e^{- rn}-7/m^2$,
$$
d(U_k)\leq\left[1-2\gamma\ell+\gamma^2 B^2\right]^kd(U_0)
$$
where 
\begin{subequations}
\begin{align*}
\ell &= \frac{\lambda_r^2}{18\| X \|_F^2}\\
B &= C n^2 \log(n r^5)^2 \lambda_r 
\end{align*}
\end{subequations}
and $d(U_k)$ is given by \eqref{eq:dist_to_soln}.  In particular, $d(U_k)$ converges to zero geometrically as long as
$\gamma < 2\ell/B^2$.
\end{thm}
For a proof of Theorem \ref{thm:main_theorem}, see Section \ref{sec:init}.

\noindent A few remarks are in order.
 
 \begin{enumerate}

\item Note that the quantity $\|X\|_F^8\lambda_r^{-4}$ is scale invariant; however, we have the bounds
$$
r^4 \leq \|X\|_F^8\lambda_r^{-4} \leq r^4(\lambda_1/\lambda_r)^4.
$$

\item One consequence of our result is that the sampling complexity is entirely independent of the desired solution tolerance. That is, the fixed set of $m \geq Cnr^6(\log n)^2$ samples suffices to produce a global solution up to arbitrary accuracy.

\item Our numerical results in \S\ref{sec:examples} suggest that in general the sampling complexity only linearly depends on the ambient dimension $n$.  Consequently a more refined analysis and initialization procedure such as that found in the recent work of \cite{candesnew15} for the case of rank-1 recovery is most likely possible also in the general rank-r recovery setting.

\item 
 This method of analysis should find use in providing recovery guarantees by gradient descent for a  broader class of nonconvex problems arising in machine learning applications such as matrix completion, nonnegative matrix factorization, clustering, etc.  More generally, such an analysis could possibly be useful towards achieving provable guarantees for machine learning problems which have many unstable saddle points, such as neural networks \cite{dauphin2014}.
 
 \end{enumerate}

\subsection{Rank-One Matrix Recovery}\label{sec:assumptions}

In the rank one setting where $x\in\mathbb{R}^n$, we can provide local convexity guarantees holding more generally for sub-gaussian measurements, subject to appropriate incoherence conditions on $x$. Suppose we receive noiseless samples of the form $y_i:=(a_i^Tx)^2$ for i.i.d. sub-gaussian vectors $\{a_i\}_{i=1}^m$, which we assume satisfy:  
\begin{equation}\label{eq:moment_conditions}
\begin{split}
\mathbb{E}[a_i]&=0\\
\mathbb{E}[a_ia_i^T]&=\Sigma\\
\end{split}
\end{equation}
where $\Sigma$ is the covariance matrix, which we assume is invertible.  With this setup, we then consider minimization of the random function
\begin{equation}\label{eq:function}
f(u):=\frac{1}{4m}\sum_{i=1}^m\left(y_i-(a_i^Tu)^2\right)^2.
\end{equation}
Note that here the Hessian is given by
\begin{equation}\label{eq:rank_one_hessian_form}
\nabla^2f(u) = \frac{1}{m}\sum_{i=1}^m\left(3(a_i^Tu)^2-(a_i^Tx)^2\right)a_ia_i^T.
\end{equation}

\noindent
We have the following convexity theorem:
\begin{thm}[Strong convexity]\label{thm:main_convexity1}Let $x\in\mathbb{R}^n$ and $\{a_i\}_{i=1}^m$ be i.i.d. sub-gaussian satisfying $\mathbb{E}[a_i]=0$ and $\mathbb{E}[a_ia_i^T]=\Sigma$.  With $\nabla^2f(x)$ as in \eqref{eq:rank_one_hessian_form}, define
\begin{subequations}
\begin{align*}
\lambda &:=  \lambda_{min}\left(\mathbb{E}\left[\nabla^2f(x)\right]\right).
\end{align*}
\end{subequations}  
If $m\geq C\|\Sigma\|_{op}^{2}n(\log n)^3,$ then with probability greater than $1-4/n^2,$ 
\begin{equation}\label{eq:main_pos_def}
(u-x)^T\nabla^2 f(u)(u-x) \geq \frac{1}{12\|\Sigma^{1/2}x\|_2^2} \lambda^2 \|\Sigma^{1/2}(u-x)\|_2^2
\end{equation}
holds uniformly for all $u\in\mathbb{R}^n$ in the ellipse around $x$ defined by
$$
\|\Sigma^{1/2}(u-x)\|_2\leq \frac{\lambda}{30\|\Sigma^{1/2}x\|_2}. 
$$
Above, $C > 0$ is a constant which depends only on the sub-gaussian norm of $a_i$.
\end{thm}

For a broad class of sub-gaussian measurements, we can provide an explicit lower bound on $\lambda_{min}\left(\mathbb{E}\left[\nabla^2f(x)\right]\right)$  thereby establishing a quantitative bound on the strong convexity parameter.   In the case of Gaussian measurements in particular, such a bound holds independent of $x$.  For more general sub-gaussian measurements, additional incoherence constraints on $x$ -- that $x$ not be ``too peaky" -- are required for strong convexity. See Lemma \ref{lem:quant_lb} for details.  While preparing this paper we became aware of related results in the thesis \cite{mahdithesis} which demonstrate a similar lower bound on the Hessian in the case of Gaussian measurements.

The finite sample convexity result holds for general sub-gaussian measurements satisfying \eqref{eq:moment_conditions}, while our initialization results require more restrictive conditions, namely that the fourth moment of the measurements is close to that of Gaussian measurements; for simplicity we have only included the result for Gaussians which follows from Lemma \ref{thm:init} in the next section.

The rest of the paper is organized as follows: in  \S\ref{sec:rankr_convexity} we prove the main \emph{finite sample convexity} result Theorem \ref{thm:main_convexity}, which relies on classifying tangent and normal directions to the manifold of solutions $\left\{XO : O\in\mathcal{O}(r)\right\}$ and an explicit formula for the expected Hessian. In \S\ref{sec:rank_one} we prove convexity results for the rank one case under more general randomness assumptions.   In \S\ref{sec:init} we prove that with high probability the initialization step produces a matrix in a convex region around the manifold of solutions and establish the convergence of gradient descent.  Briefly in \S\ref{sec:complex} we describe how our results generalize to the complex setting.  Finally, in \S\ref{sec:examples} we conclude with some numerical experiments demonstrating the performance and robustness of the results presented here.
\section{Convexity}
\subsection{General Low-Rank}\label{sec:rankr_convexity}
Here we present lemmas that are used in the proof of Theorem \ref{thm:main_convexity}, as well as a summary of the proof.  For the full proof, we refer the reader to Section \ref{sec:proof_convexity_r}.

The main lemma we rely on is the following simple characterization of the normal directions to the manifold of solutions:
\begin{lem}\label{lem:u_expansion} Assume $X$ has full column rank and let $\displaystyle O^*=\arg\min_{O\in\mathcal{O}(r)}\|XO-U\|_F^2$, which is not necessarily unique.   Then we can write
$$
UO^{*^T} = X(X^TX)^{-1}M + P_\perp U
$$
where $M\in\mathbb{R}^{r\times r}$ is a symmetric positive semidefinite matrix and $P_\perp$ is the projection onto the orthogonal complement of the column space of $X$.
\end{lem}
\begin{proof}
This basically follows from the solution to the Orthogonal Procrustes Problem \cite{procrustes}.  If we write $X^TU=ZDV^T$ for the singular value decomposition of $X^TU$, then we can expand the objective as follows:
\begin{subequations}
\begin{align*}
\arg\min_{O\in\mathcal{O}(r)}\|XO-U\|_F^2 &= \arg\min_{O\in\mathcal{O}(r)}\|X\|_F^2+\|U\|_F^2-2\langle XO, U \rangle_F\\
&=\arg\max_{O\in\mathcal{O}(r)} \langle XO, U \rangle_F\\
&=\arg\max_{O\in\mathcal{O}(r)} \langle O, ZDV^T \rangle_F\\
&=Z \left(\arg\max_{O'\in\mathcal{O}(r)} \langle O', D \rangle_F\right)V^T\\
&=ZV^T.
\end{align*}
\end{subequations}
We then find that 
$$
X^TUO^{*^T} = ZDZ^T
$$
is a symmetric positive semidefinite matrix.  As $X^TA = 0$ is equivalent to $P_\perp A = A$, we arrive at the stated claim.
\end{proof}
\noindent
This lemma says that if we consider the direction $W=U-XO^*$ between $U$ and its closest solution matrix $XO^*$ we have that
$$
O^{*^T}X^TW = O^{*^T}\left(M-X^TX\right)O^*
$$
which is a \emph{symmetric} matrix.  Why symmetry is important will become apparent after the next lemma, which establishes formulas for the expectation of the Hessian of \eqref{eq:opt_model}:

\begin{lem}\label{lem:calculus}
The gradient of $f(U) = \frac{1}{4m}\sum_{i=1}^m(y_i - a_i^TUU^Ta_i)^2$ is given by
\begin{equation}\label{eq:gradient}
\nabla f(U) = \begin{bmatrix} \nabla f_1(U) & ...& \nabla f_r(U) \end{bmatrix}\in\mathbb{R}^{n\times r}
\end{equation}
where
\begin{equation}\label{eq:gradient}
\nabla f_k(U) = \frac{1}{m}\sum_{i=1}^m (a_i^TUU^Ta_i-y_i)(a_i^Tu_k) a_i
\end{equation}
and the Hessian of $f(U)$ is given by
\begin{equation}\label{eq:hessian}
\nabla^2 f(U) = \frac{1}{m}\sum_{i=1}^m\begin{bmatrix} (a_i^TUU^Ta_i + 2(a_i^Tu_1)^2- y_i) a_ia_i^T & ... & ... \\
... & ... & ... \\
2(a_i^Tu_1)(a_i^Tu_r)a_ia_i^T & ... &  (a_i^TUU^Ta_i + 2(a_i^Tu_r)^2- y_i) a_ia_i^T
\end{bmatrix}.
\end{equation}
Moreover, if we suppose the $a_i$'s are i.i.d. centered Gaussian random vectors satisfying $\mathbb{E}[aa^T]=Id$,  then the expectation of \eqref{eq:hessian} is given by
\begin{equation}\label{eq:exp_hessian}
\mathbb{E}[\nabla^2 f(X)] = A + D
\end{equation}
where the $nr\times nr$ block matrices $A$ and $D$ satisfy
\begin{subequations}
\begin{align}
A_{ij}&=\begin{bmatrix} 2u_iu_j^T + 2u_ju_i^T+ 2(u_i^Tu_j)Id\end{bmatrix}_{n\times n}\label{eq:A}\\
D_{jj}&=\begin{bmatrix}(\|U\|_F^2-\|X\|_F^2)Id + 2(UU^T-XX^T)\end{bmatrix}_{n\times n}.
\end{align}
\end{subequations}
\end{lem}
\noindent For details, see \S\ref{sec:exp_r_proofs} in the Appendix.  We will also need a standard concentration result:
\begin{lem}\label{lem:hessian_concentration} Suppose we collect $m\geq C\delta^{-2}\beta nr\log(nr)$ samples of the form $y_i:=a_i^TXX^Ta_i$, where $\delta$ and $\beta$ are given constants and $r =$ \emph{rank}$(X)$; then we have that with probability greater than $1-2e^{-\beta rn}-6/m^2$
$$
\left\|\nabla^2f(X)-\mathbb{E}\left[\nabla^2f(X)\right]\right\|_{op} < 2\delta\|X\|_{op}^2.
$$
\end{lem}
\noindent
The sampling complexity can be improved, but we state Lemma \ref{lem:hessian_concentration} as a general proof-of-concept.  For details see \S\ref{sec:concentration_r_proofs}.
\ \\
\bigskip

To complete the proof sketch, observe that 
$$
\text{vec}(U-XO^*)^T\nabla^2f(U)\text{vec}(U-XO^*)
$$
can be written as a convex quadratic polynomial in $\|U-XO^*\|_F$, where the constant term is given by 
$$
\text{vec}(U-XO^*)^T\nabla^2f(XO^*)\text{vec}(U-XO^*)
$$
and consequently we can bound its smallest positive root using the remarks above (see \S\ref{sec:finite} for the rank one setting, where this observation is more straightforward).  We apply the concentration from above along with the following one-sided martingale bound from \cite{bentkus2003inequality} (as stated in \cite{candes2014phase}) to establish the stated non-asymptotic bound.  For details see \S\ref{sec:proof_convexity_r}.

\begin{lem}\label{lem:one_side_bent} Suppose $Y_1, Y_2, . . . , Y_m$ are i.i.d. real-valued random variables obeying $Y_i\leq b$ for some nonrandom $b > 0$, $\mathbb{E}[Y_i] = 0$, and $\mathbb{E}[Y_r^2] = v^2$. Setting $\sigma^2 = m\cdot \max(b^2,v^2)$,
$$
\mathbb{P}\left[Y_1+Y_2+...+Y_m \geq y\right]\leq \min\left\{\exp\left(-\frac{y^2}{\sigma^2}\right),c_0(1-\Phi(y/\sigma))\right\}
$$
where one can take $c_0 = 25$ and $\Phi(\cdot)$ is the CDF for the standard normal.
\end{lem}

\subsection{Rank One}\label{sec:rank_one}

In this section we restrict our attention to the setting where $x\in\mathbb{R}^n$ is a fixed unknown vector, and we receive noiseless samples of the form $y_i:=(a_i^Tx)^2$ for i.i.d. sub-gaussian vectors $\{a_i\}_{i=1}^m$, which we assume satisfy:  
\begin{equation}\label{eq:moment_conditions}
\begin{split}
\mathbb{E}[a_i]&=0\\
\mathbb{E}[a_ia_i^T]&=\Sigma\\
\end{split}
\end{equation}
where $\Sigma$ is the covariance matrix, which we assume is invertible.  Consider the eigenvalue decomposition of the covariance matrix, $\Sigma = \sum_{k=1}^n v_k v_k^T$.   An important quantity in our analysis will be 
\begin{equation}
\label{tau}
\tau(x) := \max_{1\leq k\leq n} (v_k x)^2\|\Sigma^{1/2}x\|^{-2}_2,
\end{equation}
a coherence parameter for $\Sigma^{-1/2}x,$ and 
$$
\mu_4 = \mathbb{E}[(v_k^Ta_{i})^4],
$$
a 4th moment parameter.

\subsubsection{Convexity in Expectation}\label{sec:pop}

We  consider convexity of the function $f(u)$ defined in \eqref{eq:function} (equivalently, positive semi-definiteness of the Hessian matrix $\nabla^2 f(u)$) in the neighborhood of $u=x$, first \emph{in expectation} with respect to the draw of $a_i$, or in the limit of infinitely many samples $m$.  These results are necessary for the proof of Lemma \ref{lem:quant_lb}.
\begin{lem}\label{lem:exp_formula} Assume that $\{a_i\}_{i=1}^m$ are centered sub-gaussian random vectors with $\mathbb{E}[aa^T]=\Sigma$.  Assume further that the transformed variables $b_i:=\Sigma^{-1/2}a_i$ have independent coordinates and equal fourth moment parameter $\mu_4:=\mathbb{E}[b_{ik}^4]$.  Then
\begin{equation}\label{eq:exp_hessian}
\begin{split}
\mathbb{E}[\nabla^2f(u)] &= \left(3\|\Sigma^{1/2}u\|_2^2-\|\Sigma^{1/2}x\|_2^2\right)\Sigma\\
&\ \ \ \  + \Sigma \left(6uu^T-2xx^T\right)\Sigma + (\mu_4-3)\sum_{k=1}^n\left(3(v_k^T u)^2-(v^T_kx)^2\right)v_kv_k^T.
\end{split}
\end{equation}
\end{lem}

\noindent For details, see \S\ref{sec:exp_rank1} in the Appendix.
We then have the following asymptotic convexity result:
\begin{lem}\label{lem:exp_convexity1} Let $x\in\mathbb{R}^n$ and consider the function $\mathbb{E}[f(u)]$ with $f(u)$ defined in \eqref{eq:function}.  Under the same assumptions as in Lemma \ref{lem:exp_formula} above, $\mathbb{E}[f(u)]$ is convex in the ellipse
$$
\left\{u : \|\Sigma^{1/2}(u-x)\|_2 \leq \frac{1}{3}\left(\frac{1+\min(\tau(x),1/2)[\mu_4-3]_-}{3+\tau(x)[\mu_4-3]_+}\right)\|\Sigma^{1/2}x\|_2 \right\}
$$
where $\displaystyle\tau(x)$ is the coherence of $\Sigma^{1/2}x$ as in \eqref{tau}, and above $[u]_{-} = \min\{u,0\}$ and $[u]_{+} = \max\{u,0\}$.
\end{lem}
The proof of Lemma \ref{lem:exp_convexity1} relies on the fact that $\mathbb{E}[f(x-tw)]$ is a convex quadratic polynomial in $t$.  Using this insight, we can actually bound the largest and smallest eigenvalues of the expected Hessian whenever we are within a restricted region 
$$
\frac{\delta}{3}\left(\frac{1+\tau(x)[\mu_4-3]_-}{3+\tau(x)[\mu_4-3]_+}\right)\|\Sigma^{1/2}x\|_2
$$
for some $\delta\leq1$.  In fact we find that a loose bound is given by
\begin{subequations}
\begin{align*}
\lambda_{max}\left(\mathbb{E}[\nabla^2f(u)]\right)&\leq \|\Sigma\|_{op}\|\Sigma^{1/2}x\|_2^2\left[\frac{\delta^2}{9} + 6\delta + 2(3+\tau(x)[\mu_4-3]_+)\right]\\
\lambda_{min}\left(\mathbb{E}[\nabla^2f(u)]\right)&\geq \|\Sigma^{-1}\|_{op}^{-1}\|\Sigma^{1/2}x\|_2^2\left[-2\delta(3+\tau(x)[\mu_4-3]_-)+2(1+\tau(x)[\mu_4-3]_-)\right].
\end{align*}
\end{subequations}

\begin{rem}\emph{
This result alone provides enough information to prove performance guarantees for stochastic gradient descent after an initialization procedure.  Via a union bound and covering argument, this result along with matrix concentration will also guarantee uniform convexity in this region at finite sample size $m \geq n^2$.  However, to ensure uniform convexity at finite sample size $m \geq Cn \log(n)$, we will need a more refined analysis based on the structure of the Hessian matrix, as presented in the next section. }
\end{rem}

\subsubsection{Non-Asymptotic Convexity}\label{sec:finite}
Here we present the sketch of the proof of Theorem \ref{thm:main_convexity1}.  For the full proof, we refer the reader to Section \ref{sec:proof_convexity1}.

As before, we will use the standard concentration result:
\begin{lem}\label{lem:concentration1}  Let $x\in\mathbb{R}^n$ and $\{a_i\}_{i=1}^m$ be i.i.d. sub-gaussian, satisfying \eqref{eq:moment_conditions}.  Then there exists a constant $C$ depending only on the sub-gaussian norm of $a_i$ such that if $m\geq C\epsilon^{-2}\|\Sigma\|_{op}^{2}n(\log n)^3$,  then with probability greater than $1-3/n^2$ it holds that
$$
\left\|\frac{1}{m}\sum_{i=1}^m(a_i^Tx)^2a_ia_i^T-\mathbb{E}\left[(a^Tx)^2aa^T\right]\right\|_{op}<\epsilon\|\Sigma^{1/2}x\|_2^2.
$$
\end{lem}
This result can be proved by first truncating the norms of the measurements vectors and then applying  Matrix Bernstein's Inequality (e.g., \cite{tropp2012user}).  The sampling complexity can be improved, but we state Lemma \ref{lem:concentration1} as a general proof-of-concept.  For details see \S\ref{sec:proof_concentration1}.  For $\epsilon$ sufficiently small, this result indicates that we can control the eigenvalues of $\nabla^2f(u)$ for $u$ sufficiently close to $x$. In particular, if $\nabla^2f(u)$ is positive definite in a region around $x$, then $f(u)$ is strongly convex and $x$ is the unique minimum in this region.   It is not immediately clear how to extend such control to a \emph{quantifiable} region around $x$.  However, Theorem \ref{thm:main_convexity1} requires only that we have a \emph{lower} bound on the eigenvalues.  

Assuming that $\Sigma=Id$, the same technique from \S\ref{sec:rankr_convexity} can be applied: first write $u=x+t\hat{w}$ for a unit vector $\|\hat{w}\|_2=1$ and observe that
\begin{equation*}
\begin{split}
(u-x)^T\nabla^2 f(u)(u-x) &= \frac{1}{m}\sum_{i=1}^m 3(a_i^T\hat{w})^4t^2 + 6(a_i^T\hat{w})^3(a_i^Tx)t + 2(a_i^Tx\hat{w}^Ta_i)^2\\
&=\frac{3}{m}\sum_{i=1}^m\left(A_it+B_i\right)^2 - \frac{1}{m}\sum_{i=1}^mB_i^2
\end{split}
\end{equation*}
using \eqref{eq:rank_one_hessian_form}, where we have defined
\begin{subequations}
\begin{align*}
A_i&:=(a_i^T\hat{w})^2\\
B_i&:=\left(a_i^Tx\hat{w}^Ta_i\right).
\end{align*}
\end{subequations}
Note that 
$$
\frac{1}{m}\sum_{i=1}^mB_i^2 = \frac{1}{2}(u-x)^T\nabla^2 f(x)(u-x)
$$
and consequently using Lemma \ref{lem:concentration1} and Lemma \ref{lem:exp_convexity1} we can control this term.  As before, applying Lemma \ref{lem:one_side_bent} to the positive term 
$$
\frac{3}{m}\sum_{i=1}^m\left(A_it+B_i\right)^2
$$
yields the stated conclusion.

For a broad class of sub-gaussian measurements, we can provide an explicit lower bound on $\lambda_{min}\left(\mathbb{E}\left[\nabla^2f(x)\right]\right),$ thereby establishing a quantitative bound on the strong convexity parameter. 
\begin{lem}\label{lem:quant_lb}
Suppose that $\{a_i\}_{i=1}^m$ are centered sub-gaussian random vectors with  independent coordinates, standard covariance $\mathbb{E}[aa^T]=Id,$ and equal fourth moment parameter $\mu_4:=\mathbb{E}[a_{ik}^4]$.  Then
\begin{equation}\label{eq:lb}
\lambda_{min}\left(\mathbb{E}\left[\nabla^2f(x)\right]\right)\geq 2\left(1+\min\{\tau(x),1/2\}\min\{\mu_4-3, 0 \} \right)\|x\|_2^2
\end{equation}
where $\displaystyle\tau(x) = \max_{1\leq k\leq n} \frac{|x_k|^2}{\| x \|_2^2}$ is the coherence of $x$.
\end{lem}
\noindent The proof of this uses the fact that the smallest eigenvalue is a concave function of $\mu_4$; the proof can be found in \S\ref{sec:quant_lb_proof}.  We can now quantify the lower bound appearing in \eqref{eq:main_pos_def} for a large class of sub-gaussian measurements:
\begin{enumerate}
\item \textbf{Bernoulli:}  For standard Bernoulli measurement vectors, where $a_{ik}$ are i.i.d. $\pm 1$ with equal probability, $\mu_4=1$ and we have a quantifiable strong convexity guarantee so long as $x$ is incoherent, i.e., $\tau(x)<1/2$.  This is sharp in the sense that for $x=\begin{bmatrix}1/\sqrt{2} & 1/\sqrt{2}\end{bmatrix}$ the expected Hessian has a 0 eigenvalue.
\item \textbf{Gaussian:}  For vectors $a_i$ with i.i.d. standard Gaussian entries, $\mu_4=3$ and Lemma \ref{lem:quant_lb} provides the uniform lower bound
\begin{equation}\label{eq:gaussian_lb}
\nabla^2 f(u) \succeq  \frac{1}{3}\|x\|_2^2
\end{equation}
for all $\|u-x\|_2\leq\frac{1}{15}\|x\|_2$.

\item \textbf{Sparse Gaussian:} Note that \eqref{eq:gaussian_lb} holds anytime $\mu_4\geq 3$ by Lemma \ref{lem:quant_lb}.  This includes sparse Gaussian vectors, whose coordinates are i.i.d. standard normal with probability $p$ and 0 with probability $1-p$.  In this case 
$$
\mu_4 = 3/p.
$$
\end{enumerate}

\subsection{Initialization and Gradient Descent}\label{sec:init}

We have shown that the function is strongly convex in a quantifiable region around the global minimizers.  To guarantee results for gradient descent, we will also need the following lemma which bounds the Lipschitz constant of the gradient of our function.

\begin{lem}
\label{lem:lipschitz}
Consider the function $f(U)  = \frac{1}{4m} \sum_{i=1}^m (y_i - \| a_i^T U \|_2^2 )^2$. 
Suppose $m \geq C$.   
For a universal constant $C > 0$, it holds with probability exceeding $1 - 2m^{-3}$ that 
for any $U$ within the region of convexity given by \eqref{eq:closeness},
$$
 \frac{\| \nabla f(U) - \nabla f(X) \|_F}{ \| U- XO^{*} \|_F} \leq B
$$
with $B = C n^2 \log(m)^2 \lambda_r$.  
Here, $\lambda_1 \geq \lambda_2 \geq \dots \geq \lambda_r$
are the eigenvalues of $XX^T$ and $\| U - X O^{*} \|^2_F =  \min_{O\in\mathcal{O}(r)}\|XO-U\|^2_F.$
\end{lem}

\begin{proof}
By the sub-gaussian assumption, the following holds with probability exceeding $1-2m^{-3}$:
$$
\max_{1 \leq i \leq m} \| a_i \|_2^2 \leq C n \log(m).
$$
Conditioning on this event, recalling that $\| X \|_F^2 = \lambda_1 + \dots + \lambda_r$, 
and recalling the formula for the gradient $\nabla f$ in \eqref{eq:gradient}, 
 observe the bound 
 \begin{align}
 \| \nabla f_k(U) - \nabla f_k(X) \|_2 = \| \nabla f_k(U) \|_2  &\leq \max_i  \| a_i \|_2 \left| ( a_i^T U U^T a_i - y_i)(a_i^T u_k)  \right|  \nonumber \\
&\leq \sqrt{C' n \log(m)} \max_i \left| (  a_i^T (U U^T - X O^{*} (O^{*})^T X^T) a_i)(a_i^T (u_k- v_k)) \right| \nonumber \\
&\leq  (C n \log(m))^2  \| u_k - v_k \|_2 \| U - X O^* \|_F ( \| U \|_F + \| X \|_F)  \nonumber \\
&\leq  C (n \log(m))^2 \| u_k - v_k \|_2 \frac{\lambda_r}{\| X \|_F} ( \frac{ \lambda_r}{\| X \|_F} + 2 \| X \|_F)  \nonumber \\
&\leq  C (n \log(m))^2 \| u_k - v_k \|_2 \lambda_r  \nonumber 
\end{align}
 Thus, $\| \nabla f(U) - \nabla f(X) \|_F \leq C \lambda_r n^2 \log(m)^2  \| U - X O^{*} \|_F.$
\end{proof}

  It remains to certify a point in this region to initialize gradient descent.   

\begin{lem}\label{thm:init} Suppose we take $m\geq C\beta \lambda_r^{-4}\|X\|_F^8nr^2(\log n)^2$ samples of the form \eqref{eq:sample_model}, where $\lambda_1$ and $\lambda_r$ are as in \eqref{eq:eigs}.  Define the matrix 
$$
M:=\frac{1}{2m}\sum_{i=1}^my_ia_ia_i^T
$$
and the following quantities:
\begin{subequations}
\begin{align*}
U&:=\begin{bmatrix}u_1 & u_2 & ... & u_r\end{bmatrix}_{n\times r}\\
\Sigma &:= \begin{bmatrix}\sigma_1 & 0 & ... & 0\\ 0 & \sigma_2 & ... & 0\\ ... & ... & 0 & \sigma_r \end{bmatrix}_{r\times r}-\sigma_{r+1}Id\\
U_0&:=U\Sigma^{1/2}
\end{align*}
\end{subequations}
where $\sigma_1\geq \sigma_2 ... \geq \sigma_{r+1}>0$ are the eigenvalues of $M$ and $u_i$ are the corresponding normalized eigenvectors.
Then with probability at least $1-3e^{-\beta rn}-7/m^2$ we have that
\begin{equation}\label{eq:init_close}
d(U_0) < \frac{9}{100 \| X \|_F^2 }\lambda_r^2
\end{equation}
where $d(U)$ is defined as
$$
d(U):=\min_{O\in\mathcal{O}(r)}\|XO-U\|^2_F.
$$.
\end{lem}
\noindent For the proof, see \S\ref{sec:init_proofs}.  Initializing from a matrix satisfying \eqref{eq:init_close} guarantees we are close enough so that gradient descent will converge.  We can now prove the main theorem, Theorem \ref{thm:main_theorem}:

\begin{proof}[Proof of Theorem \ref{thm:main_theorem}]
Given the number of samples $m\geq C\beta \lambda_r^{-4}\|X\|_F^8 nr (\log n),$
the following events simultaneously occur with the stated probability:

\begin{itemize}
\item Lemma \ref{thm:init} holds, and thus $d(U_0) := \min_{O\in\mathcal{O}(r)}\|XO-U_0 \|^2_F < \frac{9}{100 \| X \|_F^2}\lambda_r^2$.
\item Theorem \ref{thm:main_convexity} holds, and so considering the Taylor expansion of $f$ around $U$, the following holds for all $U$ satisfying $d(U) < \frac{9}{100 \| X \|_F^2}\lambda_r^2:$
$$
\langle U - X O^{*}, \nabla f(U) \rangle \geq \ell \| U - X O^{*} \|_F^2, \quad \quad \ell = \frac{\lambda_r^2}{18}.
$$
\item Lemma \ref{lem:lipschitz} holds with Lipschitz constant  $B = C n^2 \log(n r^5)^2 \lambda_r$.
\end{itemize}

Let $U^{+} := U_0 - \gamma \nabla f(U_0)$ and $O^{*} = \text{arg} \min_{O \in \mathcal{O}(r)} \| X O - U_0 \|_F^2$.  
Then
\begin{align}
d(U^{+}) &\leq \| U^{+} - X O^{*} \|_F^2 \nonumber \\
&= \| U_0 - \gamma \nabla f(U_0) - X O^{*} \|_F^2 \nonumber \\
&= \| U_0 - X O^{*} \|_F^2 - 2 \gamma \langle \nabla f(U_0), U_0 - X O^{*} \rangle 
+ \gamma^2 \| \nabla f(U_0) - \nabla f(X) \|_F^2 \nonumber \\
&= \| U_0 - X O^{*} \|_F^2 - 2 \gamma \langle \nabla f(U_0), U_0 - X O^{*} \rangle 
+ \gamma^2 \| \nabla f(U_0) - \nabla f(X O^{*}) \|_F^2 \nonumber \\
&\leq \| U_0 - X O^{*} \|_F^2 - 2 \gamma \ell \| U_0 - X O^{*} \|_F^2 + \gamma^2 B^2 \| U_0 - X O^{*} \|_F^2 \nonumber \\
&= [ 1 - 2 \gamma \ell + \gamma^2 B^2] \| U_0 - X O^{*} \|_F^2 \nonumber \\
&= [1 - 2 \gamma \ell + \gamma^2 B^2 ] d(U_0)
\end{align}
and, by induction, $d(U_k) \leq [1 - 2 \gamma \ell + \gamma^2 B^2 ]^k d(U_0)$.
\end{proof}

\subsection{The Complex Case}\label{sec:complex}

Suppose we are in the classical \emph{phase retrieval} setting of attempting to recover an unknown vector $z\in\mathbb{C}^n$ via Gaussian measurements of the form $a_j+ib_j$ where $a_j,b_j\sim\mathcal{N}(0,1/2Id)$.  If we write $z$ as $z=x+iw$ then a given measurement $y_j$ takes the form
\begin{subequations}
\begin{align*}
y_j &= \left|(a_j^Tx-b_j^Tw)+i(a_j^Tw+b_j^Tx)\right|^2\\
&=(a_j^Tx-b_j^Tw)^2+(a_j^Tw+b_j^Tx)^2.
\end{align*}
\end{subequations}
Note that we can cast $z$ as a matrix $Z\in\mathbb{R}^{2n\times 2}$ via
\begin{equation}\label{eq:z_matrix}
Z:=\begin{bmatrix}x & w\\ -w & x\end{bmatrix}
\end{equation}
and the measurement $y_j$ becomes
$$
\begin{bmatrix}a^T_j & b^T_j\end{bmatrix}ZZ^T\begin{bmatrix}a_j\\ b_j\end{bmatrix}
$$
where we note $\begin{bmatrix}a_j\\ b_j\end{bmatrix}\sim\mathcal{N}(0,1/2Id_{2n\times 2n})$.  Moreover, the columns of $Z$ are orthogonal, and the map
$$
e^{i\theta}\rightarrow \begin{bmatrix}\cos\theta & \sin\theta\\ -\sin\theta & \cos\theta\end{bmatrix}
$$
gives us an isomorphism onto the orientation preserving component of the orthogonal group $\mathcal{O}(2)$.  Thus this problem is equivalent to recovering an unknown real-valued rank 2 matrix, and the results in the previous sections reproduce known optimality guarantees for gradient descent in the phase retrieval model as found in e.g., \cite{candes2014phase}.  Specifically, we have shown the following:
\begin{cor} Given $m\geq C n(\log n)^2$ noiseless samples of the form
$$
y_i = |a_i^*z|^2
$$
where $z\in\mathbb{C}^n$ is an unknown vector, define 
$$
\mathbf{a}_i^T := \begin{bmatrix}\Re(a_i^T) & \Im(a_i^T)\end{bmatrix}
$$
and let $\hat{Z}$ be the output of Algorithm \ref{algo:spectral_init} applied to the data $\left\{(y_i, \mathbf{a}_i)\right\}_{i=1}^m$ with constant step size
$$
\gamma < Cn^{-4}.
$$  Then with probability at least $1-3e^{-\beta n}-7/m^2$ we have that
$$
d(U):=\min_{O\in\mathcal{O}(2)}\|ZO-\hat{Z}\|_F^2\leq C_0\left[1-\frac{\gamma\|x\|_2^4}{9}+\gamma^2C^2n^4\|x\|_2^4\right]^k
$$
where $k$ is the number of iterations of gradient descent and $Z$ is the matrix \eqref{eq:z_matrix}.
\end{cor}

\section{Examples and Experiments}\label{sec:examples}

First, we consider the performance of the algorithm  \eqref{algo:spectral_init} in the rank-1 real-valued setting, where the measurements are $y_i = (a_i^T x)^2$.  Our numerical studies strongly suggest that the algorithm \eqref{algo:spectral_init} is stable to noise, that is, given measurements of the form $y_i = (a_i^T x)^2 + \eta_i,$ the algorithm successfully returns an matrix $\widehat{x}$ up to the noise level  $\frac{\| \widehat{x} - x \|_2}{\| x\|_2} \leq \| \eta \|_2$.   We consider three different measurement ensembles:

\begin{itemize}
\item \emph{Bernoulli}: $a_i $ are i.i.d. Bernoulli random vectors
\item \emph{Standard Gaussian}: $a_i$ are i.i.d. drawn from ${\cal N}(0, Id_{n \times n})$.
\item \emph{Gaussian with covariance}: $a_i$ are i.i.d drawn from ${\cal N}(0, \Sigma)$ with covariance matrix
\end{itemize}
$$
\Sigma_{i,j} = \left\{ 
 \begin{array}{cc}
1, & i = j \\
\frac{1}{4 | i - j| }, & i \neq j
\end{array}
\right.
$$
In a first experiment, we fix an $n$-dimensional vector $x$ of unit norm with randomly-generated coefficients, and consider noiseless measurements $y_i = ( a_i^T x)^2$.  We implement the meta-algorithm \ref{algo:spectral_init}, calling Matlab's built-in function \textit{fminunc} to find a stationary point starting from the initialization.  In the local optimization procedure, we do not provide any information to fminunc other than the function itself;  by default Matlab uses a quasi-Newton method for local minimization.  We run this experiment using the three different measurement ensembles above, at problem size $n=100$ and at a number of measurements $m = 2n, 3n, \dots, 8n$.  If the solution $\widehat{x}$ recovered by the algorithm is within the tolerance $\min\{ \| \widehat{x} - x \|_2, \| \widehat{x} + x \|_2 \} \leq .001$, we say the algorithm has succeeded in finding the global solution.  In Figure \ref{fig1}, the results of this experiment are displayed, averaged over 100 trials. 

\bigskip

Next, we analyze numerically the stability of the algorithm to additive measurement noise.  For these experiments, we consider noisy measurements of the form 
$$
y_i = ( a_i^T x)^2 + \eta_i
$$
where $\eta_i$ are i.i.d. mean-zero uniformly distributed, and normalized such that $\| \eta \|_2 = \mu \|  \sum_i  ( a_i^T x)^2 \|_2$ for $\mu = .5$ (low signal to noise ratio) and $\mu = 2$ (high signal to noise ratio).  We observe that the meta-algorithm is robust to such additive noise, with relative reconstruction error $ \min\{ \| \widehat{x} - x \|_2, \| \widehat{x} + x \|_2 \}$ averaging below the signal to noise threshold.  We leave a theoretical analysis of this observed noise stability to future work.

\begin{figure}
\begin{center}
\includegraphics[width=8cm]{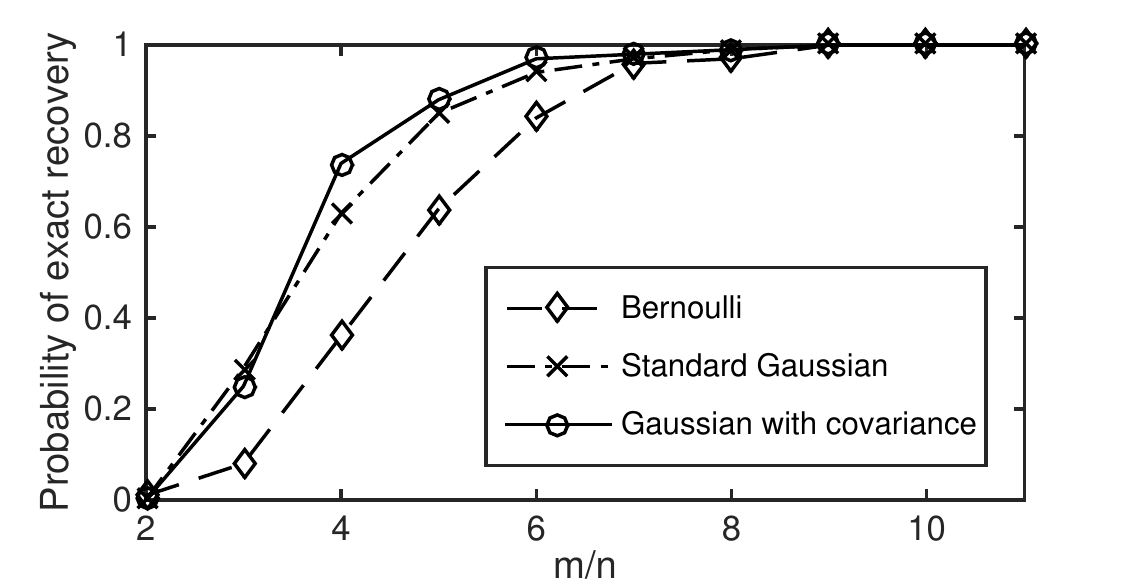}
\caption{Phase transitions for exact recovery via Algorithm Initialize and Descend  \eqref{algo:spectral_init} using different random measurement ensembles.}
\label{fig1}
\end{center}
\end{figure}

\begin{figure}
\begin{center}
\includegraphics[width=8cm,height=4.3cm]{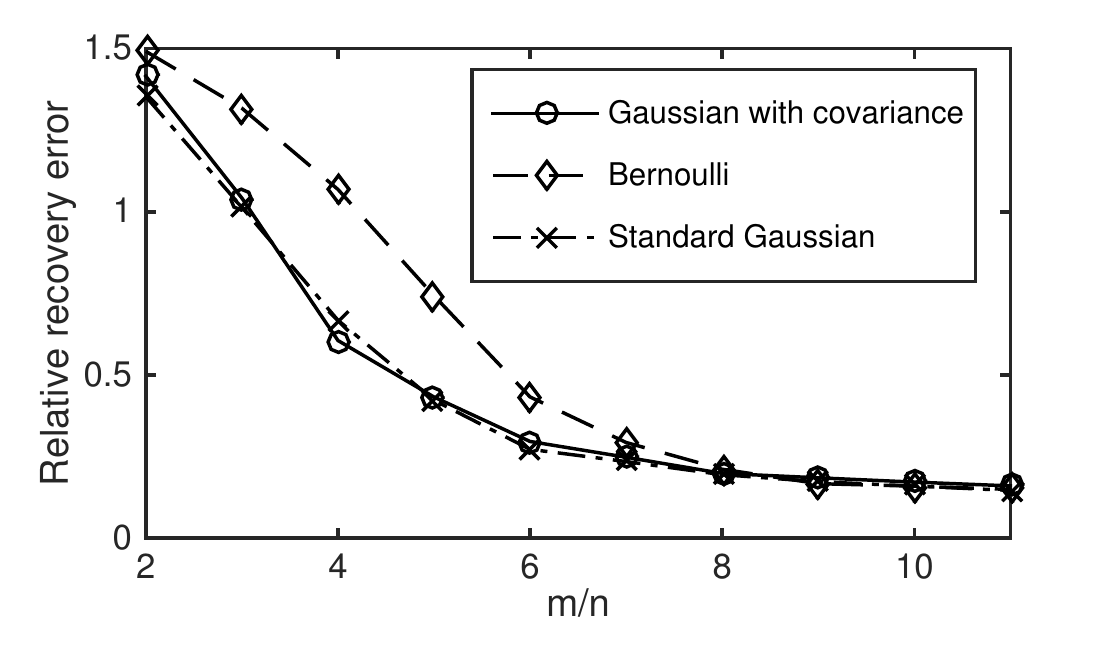} 
\includegraphics[width=8.2cm,height=4.5cm]{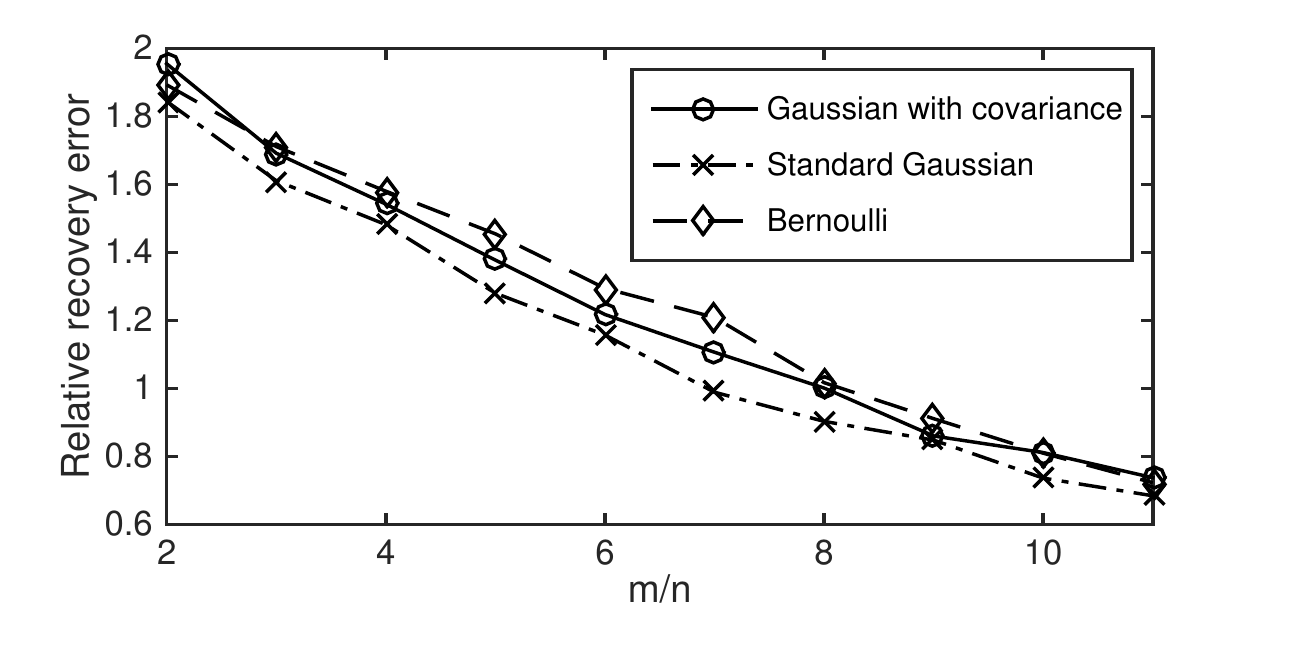}
\caption{Performance of Algorithm Initialize and Descend  \eqref{algo:spectral_init} in the presence of additive noise $y_i = ( a_i^T x)^2 + \eta_i$ at different signal-to-noise levels. Right: $\| \eta \|_2 / \sum_i ( a_i^T x)^2 = .5$ and Left: $\| \eta \|_2 /\sum_i ( a_i^T x)^2 = 2$.}
\label{fig2}
\end{center}
\end{figure}

Finally, we test the performance of Algorithm  \ref{algo:spectral_init} in the more general rank-$r$ case.  We consider noiseless measurements $y_i = a_i^T X X^T a_i$  where the $a_i$ are i.i.d. standard Gaussian and $X \in \mathbb{R}^{n \times 5}$ is a rank-5 matrix with orthogonal columns, normalized so that $\| X \|_{F} = 1$.   We implement Algorithm \ref{algo:spectral_init}, calling Matlab's built-in function \textit{fminunc} to find a stationary point starting from the initialization.  In the local optimization procedure, we do not provide any information to fminunc other than the function itself;  by default Matlab uses a quasi-Newton method for local minimization.  We run this experiment at problem size $n=20$ and $m=5n, 10n, \dots, 25n$.    We declare the algorithm to have converged to global solution if the matrix $\widehat{X} \in \mathbb{R}^{n \times 5}$ recovered by the algorithm satisfies
$$
\| X Z V^T - \widehat{X} \|_{F} = \min_{O\in\mathcal{O}(r)}\| XO- \widehat{X} \|_F \leq .001
$$
where $Z \Sigma V^T = X^T \widehat{X}$ is the singular value decomposition.n In Figure \ref{fig1}, the results of the experiment are displayed, averaged over 100 trials. 

\begin{figure}
\begin{center}
\includegraphics[width=7cm]{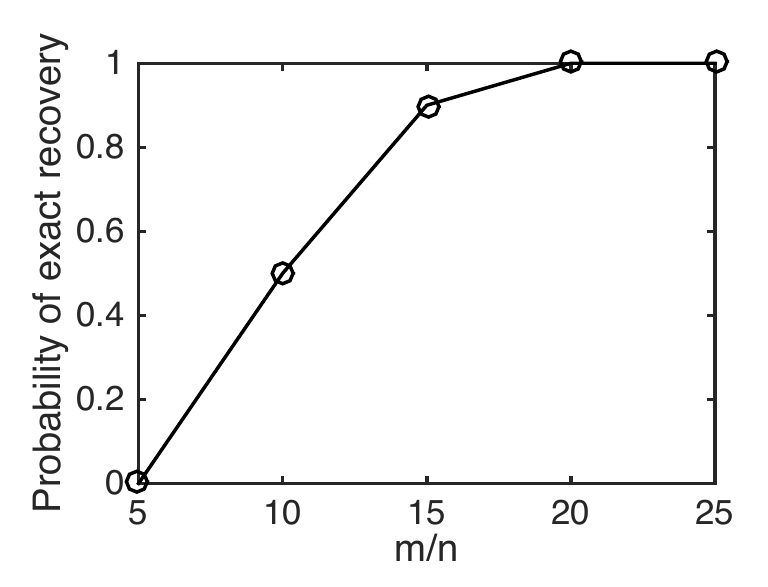}
\caption{Phase transitions for exact recovery via Algorithm Initialize and Descend  \eqref{algo:spectral_init} in recovering a rank-5 matrix from quadratic Gaussian measurements.}
\label{fig3}
\end{center}
\end{figure}

\subsection*{Acknowledgements} 
R. Ward and C. White were funded in part by an NSF CAREER Grant and an AFOSR Young Investigator Award. 
We would like to thank Ju Sun for pointing out a mistake in the original rank one proof, and thank Mahdi Soltanolkotabi and Laurent Jacques for additional helpful comments and corrections.   C. White would like to thank Aaron Royer and Ravi Srinivasan for helpful conversations.


\bibliographystyle{myalpha}
\newcommand{\etalchar}[1]{$^{#1}$}

\section{Appendix}
\subsection{Proofs for \S\ref{sec:rankr_convexity}}
\subsubsection{Proof of Lemma \ref{lem:calculus}}\label{sec:exp_r_proofs}
\begin{proof}
The proofs of \eqref{eq:gradient} and \eqref{eq:hessian} use standard vector calculus.  For \eqref{eq:exp_hessian}, we use the fact that for any vector $x\in\mathbb{R}^n$
$$
\mathbb{E}[(x^Ta)^2aa^T] = \|x\|_2^2Id + 2xx^T
$$
which can be seen by writing out the entries of the matrix individually.  This implies
\begin{equation*}\label{eq:matrix_expectation}
\begin{split}
\mathbb{E}[(a^TXX^Ta)aa^T] &= \sum_{i=1}^r\mathbb{E}[(a^Tx_i)^2aa^T]\\
&=\sum_{i=1}^r \left(\|x_i\|_2^2Id + 2x_ix_i^T\right)
\end{split}
\end{equation*}
which, combined with
$$
\mathbb{E}[(a^Tx_i)(a^Tx_j)aa^T] = x_ix_j^T + x_jx_i^T+ x_i^Tx_jId
$$
yields the stated result.
\end{proof}

\subsubsection{Proof of Lemma \ref{lem:hessian_concentration}}\label{sec:concentration_r_proofs}

We begin with a more general concentration result.

\begin{thm}\label{theorem:tensor_concentration}
Let $X\in\mathbb{R}^{n\times r}$ be a given matrix with orthogonal columns; suppose $m\geq C\delta^{-2}\beta nr\log(n)$ where $\delta$ and $\beta$ are given constants and $r=$\emph{rank}$(X)$.  Then we have that with probability greater than $1-2e^{-\beta rn}-6/m^2$
$$
\left\|\frac{1}{m}\sum_{i=1}^ma_ia_i^T\otimes a_ia_i^T-\mathbb{E}[a_ia_i^T\otimes a_ia_i^T]\right\|_{\big|_{X\otimes\mathbb{R}^n}} < \delta
$$
where $\|\cdot\|_{\big|_{X\otimes\mathbb{R}^n}}$ is the operator norm of the matrix restricted to the subspace spanned by the columns of $X$ tensored with $\mathbb{R}^n$.
\end{thm}
\begin{proof}
Let $x\otimes z\in X\otimes\mathbb{R}^n$ be an arbitrary unit vector.  Write $z=z_{\perp} + w$, where $w\in X$ and $\langle z_{\perp}, x\rangle = 0$.  We must consider the quantity
\begin{equation*}\label{eq:components}
\begin{split}
\big|\frac{1}{m}&\sum_{i=1}^m(a_i^Tx)^2(a_i^Tz)^2-2(x^Tz)^2-1\big|\\
&=\big|\frac{1}{m}\sum_{i=1}^m(a_i^Tx)^2(a_i^Tz_{\perp})^2-\|z_{\perp}\|_2^2+\frac{1}{m}\sum_{i=1}^m(a_i^Tx)^2(a_i^Tw)^2-2(x^Tw)^2-\|w\|_2^2\\
&\ \ \ \ +\frac{2}{m}\sum_{i=1}^m(a_i^Tx)^2(a_i^Tz_{\perp})(a_i^Tw)\big|.
\end{split}
\end{equation*}
\textbf{First Term.} Note that we have a product of independent subexponential random variables, and so if we condition on the bounds
\begin{equation*}\label{eq:first_cond}
\begin{split}
\frac{1}{m}\sum_{i=1}^m (a_i^Tx)^2 &< 3\\
\max_{1\leq i\leq m} (a_i^Tx)^2 &<9\log m
\end{split}
\end{equation*}
both of which happen with probability at least $1-1/m^2$, we find via Bernstein that
$$
\mathbb{P}[\big|\frac{1}{m}\sum_{i=1}^m(a_i^Tx)^2\big((a_i^Tz_{\perp})^2-\|z_{\perp}\|_2^2\big)\big|>\delta/6]\leq 2\exp\big(-c\min\{\frac{m\delta^2}{108},\frac{m\delta}{54\log m}\}\big)
$$
which we can make smaller than $e^{-\alpha rn}$ so long as $m\geq C\delta^{-2}\alpha nr\log(n)$ and we conclude via an $\epsilon$-net argument that
\begin{equation}\label{eq:first_term}
\mathbb{P}[\big\|\frac{1}{m}\sum_{i=1}^ma_ia_i^T\otimes (a_ia_i^TP_{\perp})-P_{\perp}\big\|_{\big|_{X\otimes\mathbb{R}^n}}>\delta/3]\leq e^{-\beta rn} + 2/m^2
\end{equation}
where $P_{\perp}$ is the orthogonal projection onto the complement of $X$ in $\mathbb{R}^n$.\ \\
\ \\
\textbf{Third Term.} We recognize $(a_i^Tz_{\perp})$ as a sub-gaussian random variable, and thus if we condition on
\begin{equation}\label{eq:third_cond}
\frac{1}{m}\sum_{i=1}^m(a_i^Tx)^4(a_i^Tw)^2\leq 16
\end{equation}
we find via Hoeffding that
$$
\mathbb{P}[\big|\frac{2}{m}\sum_{i=1}^m(a_i^Tx)^2(a_i^Tz_{\perp})(a_i^Tw)\big|>\delta/3]\leq\exp\big(1-\frac{c\delta^2m}{16}\big).
$$
We can make this bound smaller than $e^{-\alpha rn}$ so long as $m\geq C\delta^{-2}\alpha rn$. As \eqref{eq:third_cond} happens with probability greater than $1-1/m^2$ we find
\begin{equation}\label{eq:third_term}
\mathbb{P}[\big\|\frac{2}{m}\sum_{i=1}^ma_ia_i^T\otimes (Pa_ia_i^TP_{\perp})\big\|_{\big|_{X\otimes\mathbb{R}^n}}>\delta/3] \leq e^{-\beta rn} + 1/m^2.
\end{equation}\ \\
\ \\
\textbf{Second Term.}  For this term we further decompose $w$ into its $x$-component and its $x_{\perp}$-component.  We can apply the same analysis for the first and third terms to the $x_{\perp}$ terms, and have only to deal with
$$
\mathbb{P}[\big|\frac{(x^Tw)^2}{m}\sum_{i=1}^m(a_i^Tx)^4-3(x^Tw)^2\big|>\delta/3].
$$
If we condition on 
$$
\max_{1\leq i\leq m}(a_i^Tx)^4 \leq 81(\log m)^2
$$
and note that
$$
\big|3-\mathbb{E}\big[(a_i^Tx)^4\big|(a_i^Tx)^4\leq 81(\log m)^2\big]\big| \leq 1/m^2
$$
we find via Hoeffding that
\begin{equation}\label{eq:second_hoeff}
\mathbb{P}[\big|\frac{(x^Tw)^2}{m}\sum_{i=1}^m(a_i^Tx)^4-\mathbb{E}[...]\big|>\delta/3]\leq\exp\big(\frac{-\delta^2m^2}{Cm + 81(\log m)^2\delta m}\big)
\end{equation}
which we can make smaller than $e^{-\alpha r}$ so long as $m\geq C\alpha \delta^{-2}r(\log r)^2$.  Moreover, note that we can also make \eqref{eq:second_hoeff} smaller than $1/m^2$ for $m\geq C\delta^{-2}$.  We conclude that
\begin{equation}\label{eq:second_term}
\mathbb{P}[\big\|\frac{1}{m}\sum_{i=1}^ma_ia_i^T\otimes a_ia_i^T-\mathbb{E}[a_ia_i^T\otimes a_ia_i^T]\big\|_{\big|_{X\otimes X}}>\delta/3]\leq 3\min\{e^{-\beta r},1/m^2\}+3/m^2.
\end{equation}
\noindent
Combining \eqref{eq:first_term}, \eqref{eq:second_term}, and \eqref{eq:third_term} yields the stated result.
\end{proof}

\begin{cor}\label{cor:full_concentration} Suppose we collect $m\geq C\delta^{-2}\beta nr\log(n)^2$ samples of the form $y_i:=a_i^TXX^Ta_i$, where $\delta$ and $\beta$ are given constants and $r =$ \emph{rank}$(X)$; then we have that with probability greater than $1-2e^{-\beta rn}-6/m^2$
$$
\left\|\frac{1}{m}\sum_{i=1}^my_ia_ia_i^T-\|X\|_F^2Id -2X^*X^{*^T}\right\|_{op} < \delta\|X\|_F^2.
$$
\end{cor}
\begin{proof} Note that $y_i = \sum_{k=1}^r(a_i^Tx_k)^2$, and so by Theorem \ref{theorem:tensor_concentration} we find that with probability greater than $1-2e^{-\beta rn}-6/m^2$
$$
\left\|\frac{1}{m}\sum_{i=1}^m(a_i^Tx_k)^2a_ia_i^T-\|x_k\|_2^2Id -2x_kx_k^T\right\|_{op} < \delta\|x_k\|_2^2
$$
and so
\begin{equation*}
\begin{split}
\left\|\frac{1}{m}\sum_{i=1}^my_ia_ia_i^T-\|X\|_F^2Id -2XX^T\right\|_{op}&\leq \sum_{k=1}^r\big\|\frac{1}{m}\sum_{i=1}^m(a_i^Tx_k)^2a_ia_i^T-\|x_k\|_2^2Id -2x_kx_k^T\big\|_{op}\\
\ \\
&\leq \delta\|X\|_F^2.
\end{split}
\end{equation*}
\end{proof}

\begin{cor}\label{cor:hessian_concentration} Suppose we collect $m\geq C\delta^{-2}\beta nr\log(n)^2$ samples of the form $y_i:=a_i^TXX^Ta_i$, where $\delta$ and $\beta$ are given constants and $r =$ \emph{rank}$(X)$; then we have that with probability greater than $1-2e^{-\beta rn}-6/m^2$
$$
\left\|\nabla^2f(X)-\mathbb{E}\left[\nabla^2f(X)\right]\right\|_{op} < 2\delta\|X\|_{op}^2.
$$
\end{cor}
\begin{proof} Note that 
$$\nabla^2f(X) = \frac{2}{m}\sum_{i=1}^m\left[X^Ta_ia_i^TX\right]\otimes a_ia_i^T$$
and so we can use Theorem \ref{theorem:tensor_concentration} to write
\begin{equation}
\begin{split}
\left\|X^T\otimes Id\left(\frac{1}{m}\sum_{i=1}^ma_ia_i^T\otimes a_ia_i^T -\mathbb{E}[...]\right)X\otimes Id\right\|_{op} &< \delta\|X\|_{op}^2.
\end{split}
\end{equation}

\end{proof}

\subsubsection{Proof of the Convexity Theorem \ref{thm:main_convexity}}\label{sec:proof_convexity_r}

We will rely on the following Lemma from \cite{bentkus2003inequality}, as stated in \cite{candes2014phase}: 
\begin{lem}\label{lem:candes}Suppose $Y_1, Y_2, . . . , Y_m$ are i.i.d. real-valued random variables obeying $Y_i\leq b$ for some nonrandom $b > 0$, $\mathbb{E}[Y_i] = 0$, and $\mathbb{E}[Y_r^2] = v^2$. Setting $\sigma^2 = m\cdot \max(b^2,v^2)$,
$$
\mathbb{P}\left[Y_1+Y_2+...+Y_m \geq y\right]\leq \min\left\{\exp\left(-\frac{y^2}{\sigma^2}\right),c_0(1-\Phi(y/\sigma))\right\}
$$
where one can take $c_0 = 25$ and $\Phi(\cdot)$ is the CDF for the standard normal.
\end{lem}

\begin{proof}[Proof of Theorem \ref{thm:main_convexity}]  Let $\hat{W}:=W/\|W\|_F$ be the normalized direction from $U$ to $XO^*$ and let $t\geq 0$ be a positive scalar.  Moreover, WLOG we will be assuming that $\|X\|_F=1$. 

Finally, because $f(U)$ is invariant under the action of $\mathcal{O}(r),$ it suffices to consider the case where $O^*=Id$.

Consider the single-variable function $f(X+t\hat{W})$ which can be written
$$
f(t) = \frac{1}{4m}\sum_{i=1}^m\left(2ta_i^TX\hat{W}^Ta_i+t^2a_i^T\hat{W}\hat{W}^Ta_i\right)^2.
$$
It is straightforward to verify that
\begin{equation}\label{eq:poly_second}
f''(t) = \frac{1}{m}\sum_{i=1}^m 3\left(a_i^T\hat{W}\hat{W}^Ta_i\right)^2t^2 + 6\left(a_i^T\hat{W}\hat{W}^Ta_i\right)\left(a_i^TX\hat{W}^Ta_i\right)t + 2\left(a_i^TX\hat{W}^Ta_i\right)^2
\end{equation}
which is a convex polynomial in $t$; observe that $f''(0)>0$.   If the linear term is positive, then clearly \eqref{eq:poly_second} is positive for all $t$ and we have nothing to show (the smallest eigenvalue is bounded below by $f''(0)$ in the direction $\hat{W}$).
Define the following quantities:
\begin{subequations}
\begin{align*}
A_i&:=\left(a_i^T\hat{W}\hat{W}^Ta_i\right) = \|\hat{W}^Ta_i\|_2^2\\
B_i&:=\left(a_i^TX\hat{W}^Ta_i\right) = \langle X^Ta_i, \hat{W}^Ta_i \rangle
\end{align*}
\end{subequations}
and note that \eqref{eq:poly_second} can be written as
$$
f''(t) = \frac{3}{m}\sum_{i=1}^m\left(A_it+B_i\right)^2 - \frac{1}{m}\sum_{i=1}^mB_i^2.
$$
Consider the random variable
$$
Z_i(t):=\left(A_it+B_i\right)^2\geq 0
$$
Observe that $A_i$ is a chi-squared random variable with 1 degree of freedom by the normalization $\|\hat{W}\|_F=1$.  Thus we have
\begin{subequations}
\begin{align*}
\mathbb{E}[A_i] &= 1\\
\mathbb{E}[A_i^2] &= 3\\
\mathbb{E}[A_i^4] &= 105\\
\mathbb{E}[A_i^6] &= 10395\\
\mathbb{E}[A_i^8] &= 2027025\\
\mathbb{E}[A_i^{12}] &= 18602008425.
\end{align*}
\end{subequations} By the definition of $f''(t)$, we have 
$$
f''(0) = \text{vec}(\hat{W})^T\mathbb{E}\left[\nabla^2f(X)\right]\text{vec}(\hat{W})
$$
and so
\begin{equation}\label{eq:B_hessian}
\mathbb{E}[B_i^2] = \frac{1}{2}\text{vec}(\hat{W})^T\mathbb{E}\left[\nabla^2f(X)\right]\text{vec}(\hat{W}).
\end{equation}
Moreover, 
\begin{equation*}
\begin{split}
\mathbb{E}[A_iB_i] &= \mathbb{E}\left[\left(\sum_{k=1}^r(a_i^Tw_k)^2\right)\left(\sum_{k=1}^r(a_i^Tx_k)(a_i^Tw_k)\right)\right]\\
&=\mathbb{E}\left[\sum_{k,q=1}^r(a_i^Tw_k)^2(a_i^Tx_q)(a_i^Tw_q)\right]\\
&=\sum_{k,q=1}^r 2(w_k^Tx_q)(w_k^Tw_q) + \|w_k\|_2^2x_q^Tw_q\\
&= 2\sum_{k,q=1}^r x_q^Tw_kw_k^Tw_q+ \sum_{q=1}^rx_q^Tw_q\\
&=2\sum_{q=1}^rx_q^T\hat{W}\hat{W}^Tw_q+ \text{tr}(X^T\hat{W})\\
&=2\text{tr}(X^T\hat{W}\hat{W}^T\hat{W})+ \text{tr}(X^T\hat{W})\\
\end{split}
\end{equation*}
We then have that the mean $\mu(t):=\mathbb{E}[Z_i(t)]$ is given by
\begin{equation}\label{eq:z_mean}
3t^2 + \left(4\text{tr}(X^T\hat{W}\hat{W}^T\hat{W})+ 2\text{tr}(X^T\hat{W})\right)t + \frac{1}{2}\text{vec}(\hat{W})^T\mathbb{E}\left[\nabla^2f(X)\right]\text{vec}(\hat{W}).
\end{equation}
\noindent
Next we consider the variance of $Z_i(t)$:
\begin{equation*}
\setlength{\jot}{10pt}
\begin{split}
\mathbb{E}[\left(Z_i(t)-\mu(t)\right)^2] &\leq \mathbb{E}[Z^2_i(t)]\\
&=\mathbb{E}\left[A_i^4\right]t^4 + 4\mathbb{E}\left[A_i^3B_i\right]t^3 + 6\mathbb{E}\left[A_i^2B^2_i\right]t^2 + 4\mathbb{E}\left[A_iB^3_i\right]t + \mathbb{E}\left[B_i^4\right]\\
&=105t^4 + 4\mathbb{E}\left[A_i^3\langle X^Ta_i,\hat{W}^Ta_i\rangle\right]t^3\\
&\ \ \ \  + 6\mathbb{E}\left[A_i^2\langle X^Ta_i,\hat{W}^Ta_i\rangle^2\right]t^2 + 4\mathbb{E}\left[A_i\langle X^Ta_i,\hat{W}^Ta_i\rangle^3\right]t + \mathbb{E}\left[\langle X^Ta_i,\hat{W}^Ta_i\rangle^4\right]\\
&\leq105t^4 + 4\sqrt{\mathbb{E}[A_i^6]\mathbb{E}\left[\langle X^Ta_i,\hat{W}^Ta_i\rangle^2\right]}t^3\\
&\ \ \ \  + 6\mathbb{E}[A_i^3\|X^Ta_i\|_2^2]t^2 + 4\sqrt{\mathbb{E}[A_i^2]\mathbb{E}[\langle X^Ta_i,\hat{W}^Ta_i\rangle^6]}t + \mathbb{E}\left[\|X^Ta_i\|_2^4\|\hat{W}^Ta_i\|_2^4\right]
\end{split}
\end{equation*}
where we used either Cauchy-Schwarz or H\"older's inequality on each term.  Proceeding with applications of H\"older and Cauchy-Schwarz, recognizing that $\|X^Ta_i\|_2^2$ is also Chi-squared with one degree of freedom under the assumption that $\|X\|_F=1$, and noting from \eqref{eq:B_hessian} that $\mathbb{E}[B_i^2]\leq 3$ we have
\begin{equation*}
\setlength{\jot}{10pt}
\begin{split}
&=105t^4 + 4\sqrt{10395\cdot\mathbb{E}\left[B_i^2\right]}t^3\\
&\ \ \ \  + 6\mathbb{E}[A_i^3\|X^Ta_i\|_2^2]t^2 + 4\sqrt{3\cdot\mathbb{E}[\langle X^Ta_i,\hat{W}^Ta_i\rangle^6]}t + \mathbb{E}\left[\|X^Ta_i\|_2^4\|\hat{W}^Ta_i\|_2^4\right]\\
&\leq 105t^4 + 707t^3 + 6\sqrt{\mathbb{E}[A_i^6\|X^Ta_i\|_2^4]}t^2 + 4\sqrt{3\cdot\mathbb{E}[\|X^Ta_i\|_2^6\|\hat{W}^Ta_i\|_2^6]}t + \sqrt{\mathbb{E}\left[\|X^Ta_i\|_2^8]\mathbb{E}[\|\hat{W}^Ta_i\|_2^8\right]}\\
&\leq 105t^4 + 707t^3 + 6\left(\mathbb{E}[A_i^{12}]\mathbb{E}[\|X^Ta_i\|_2^8]\right)^{1/4}t^2 + 4\sqrt{3}\left(\mathbb{E}[\|X^Ta_i\|_2^{12}]\mathbb{E}[\|\hat{W}^Ta_i\|_2^{12}]\right)^{1/4}t + 105\\
&\leq 105t^4 + 707t^3 + 7094t^2 + 707t + 105\\
&= C(t)^2.
\end{split}
\end{equation*}
Observe that the mean can be bounded as
$$
\mu(t) \leq 3t^2 + 6t + 3\lambda_1.
$$
Now define
\begin{subequations}
\begin{align*}
Y_i(t) &:= \mu(t)-Z_i(t)\\
b(t) &:= 3t^2 + 6t + 3\lambda_1\\
v^2(t) &:= C(t)^2\\
\sigma^2(t) &:= mC(t)^2\\
y &:= m\lambda_r/12.
\end{align*}
\end{subequations}
Applying Lemma \ref{lem:candes} above yields
$$
\mathbb{P}\left[\mu(t)-\frac{1}{m}\sum_{i=1}^mZ_i(t) \geq \lambda_r/12\right]\leq \min\left\{\exp\left(-\frac{\lambda_r^2m}{144C(t)^2}\right),25\left(1-\Phi(\frac{\lambda_r\sqrt{m}}{12C(t)})\right)\right\}.
$$
Using the well-known bound 
$$
1-\Phi(\frac{\lambda_r\sqrt{m}}{12C(t)}) < \frac{12C(t)}{\lambda_r\sqrt{2m\pi}}\exp\left(-\frac{\lambda_r^2m}{288C(t)^2}\right)
$$
we find that if $m\geq 288\alpha\lambda_r^{-2}C(t)^2nr$ then with probability at least $1-e^{-\alpha nr}$ we have
\begin{equation}\label{eq:e_net_lb}
\begin{split}
f''(t) &= \frac{3}{m}\sum_{i=1}^m\left(A_it+B_i\right)^2 - \frac{1}{m}\sum_{i=1}^mB_i^2\\
&\geq 3\mu(t) -\lambda_r/4 - \frac{1}{m}\sum_{i=1}^mB_i^2\\
&\geq 2t^2 -6t -\lambda_r/4 + \frac{3}{2}\text{vec}(\hat{W})^T\mathbb{E}\left[\nabla^2f(X)\right]\text{vec}(\hat{W})-\frac{1}{2}\text{vec}(\hat{W})^T\nabla^2f(X)\text{vec}(\hat{W})
\end{split}
\end{equation}
where we used \eqref{eq:z_mean} to lower bound $\mu(t)$ by
$$
t^2 -6t + \frac{1}{2}\text{vec}(\hat{W})^T\mathbb{E}\left[\nabla^2f(X)\right]\text{vec}(\hat{W})
$$
and the fact that 
$$
\frac{1}{m}\sum_{i=1}^mB_i^2 = \frac{1}{2}\text{vec}(\hat{W})^T\nabla^2f(X)\text{vec}(\hat{W}).
$$
Moreover, an $\epsilon$-net argument over all directions $\hat{W}$ shows that \eqref{eq:e_net_lb} holds for an arbitrary $\hat{W}$ with probability at least $1-e^{-\beta nr}$.

Further observe that our condition on $m$ guarantees 
$$
\left\|\nabla^2f(X)-\mathbb{E}\left[\nabla^2f(X)\right]\right\|_{op} < \frac{\lambda_r}{4}
$$
with probability at least $1-2e^{-\beta rn}-6/m^2$ by Corollary \ref{cor:hessian_concentration}.   This implies that 
$$
\frac{3}{2}\text{vec}(\hat{W})^T\mathbb{E}\left[\nabla^2f(X)\right]\text{vec}(\hat{W})-\frac{1}{2}\text{vec}(\hat{W})^T\nabla^2f(X)\text{vec}(\hat{W})>\frac{15}{8}\lambda_r
$$
so that 
$$
f''(t) \geq 2t^2 -6t +\frac{15}{8}\lambda_r
$$
with probability at least $1-3e^{-\beta rn}-6/m^2$ for any direction $\hat{W}$.  Thus by a tangent line bound we find that the smallest positive root of $f''(t)$ is bounded below by
$$
t^* \geq \frac{15}{48}\lambda_r > \frac{3}{10}\lambda_r
$$
and we note that
\begin{subequations}
\begin{align*}
f''(\frac{15}{48}\lambda_r) &= 2(\frac{15}{48})^2\lambda_r^2 + \frac{15}{8}(\lambda_r-\lambda_r)\\\
&\geq \frac{\lambda_r^2}{18}.
\end{align*}
\end{subequations}
Thus $f''(t)\geq \frac{\lambda_r^2}{18}$ for all $t\in[0,\frac{3}{10}\lambda_r]$ which is the advertised lower bound.

For the upper bound, observe that by Cauchy-Schwarz
$$
\frac{6}{m}\sum_{i=1}^m\left(a_i^T\hat{W}\hat{W}^Ta_i\right)\left(a_i^TX\hat{W}^Ta_i\right) \leq 6\sqrt{\frac{1}{m}\sum_{i=1}^m\left(a_i^T\hat{W}\hat{W}^Ta_i\right)^2}\sqrt{\frac{1}{m}\sum_{i=1}^m\left(a_i^TX\hat{W}^Ta_i\right)^2}
$$
and thus we find an upper bound for \eqref{eq:poly_second} is given by
\begin{equation}\label{eq:easy_poly_lb}
f''(t)\leq 3a^2t^2 +6abt + 2b^2
\end{equation}
where
\begin{subequations}
\begin{align*}
a &= \sqrt{\frac{1}{m}\sum_{i=1}^m\left(a_i^T\hat{W}\hat{W}^Ta_i\right)^2}\\
b &= \sqrt{\frac{1}{m}\sum_{i=1}^m\left(a_i^TX\hat{W}^Ta_i\right)^2}.
\end{align*}
\end{subequations}

Now, $\omega_i:=\left(a_i^T\hat{W}\hat{W}^Ta_i\right)$ is a chi-squared random variable with one degree of freedom; consequently we find
$$
\mathbb{P}[\frac{1}{m}\sum_{i=1}^m\omega_i^2 \geq C^2n^2r^2] \leq \frac{m}{\sqrt{Cnr}}e^{-Cnr} \leq e^{-\tilde{C}nr}
$$
as long as $m\geq Cnr$.  Consequently an $\epsilon$-net argument shows us that for \emph{any} $\hat{W}\in\mathbb{R}^{n\times r},\ \|\hat{W}\|_F=1$ we have
$$
\frac{1}{m}\sum_{i=1}^m\left(a_i^T\hat{W}\hat{W}^Ta_i\right)^2 \leq C^2n^2r^2
$$
with probability greater than $1-e^{-\beta nr}$.

Returning to \eqref{eq:easy_poly_lb} we find then that
\begin{equation*}\label{eq:plug_in}
\begin{split}
f''(t) &\leq 3\left(at + b\right)^2 - b^2\\
&\leq 3\left(Cnr\lambda_r + b\right)^2
\end{split}
\end{equation*}
for all $t\in[0,\frac{3}{10}\lambda_r]$.

To finish the proof, observe that we can write
$$
b = \frac{1}{\sqrt{2}}\sqrt{\text{vec}(\hat{W})^T\nabla^2f(X)\text{vec}(\hat{W})}
$$
and by Corollary \ref{cor:hessian_concentration} (where, given the number of measurements $m$, we may take $\delta \geq C \lambda_r / \lambda_1$) we have
\begin{subequations}
\begin{align*}
b &\leq \frac{1}{\sqrt{2}}\sqrt{\text{vec}(\hat{W})^T\mathbb{E}\left[\nabla^2f(X)\right]\text{vec}(\hat{W}) + 2\lambda_r}\\
&\leq \frac{1}{\sqrt{2}}\sqrt{2\lambda_r}
\end{align*}
\end{subequations}
yielding
\begin{equation*}\label{eq:last_up}
\begin{split}
f''(t) &\leq C( n^2r^2\lambda_r^2 + \lambda_r) .
\end{split}
\end{equation*}

From everything above, we conclude that for $\frac{ \| U \|_F}{\| X \|_F} \leq \frac{3}{10}\lambda_r\left( \frac{X X^T}{\| X \|_F^2} \right)$, 
\begin{align}
\frac{1}{18} \lambda_r^2 \left( \frac{X X^T}{\| X \|_F^2} \right) \left\| \frac{U - X}{\| X \|_F} \right\|_F^2 &\leq \text{vec}\left( \frac{U - X}{\| X \|_F} \right)^T \nabla^2 f \left( \frac{U}{\| X \|_F}; \frac{X}{\| X \|_F} \right) \text{vec}\left( \frac{U - X}{\| X \|_F} \right) \nonumber \\
&\leq C \left( n^2r^2\lambda_r^2 \left( \frac{X X^T}{\| X \|_F^2}  \right) + \lambda_r \left( \frac{X X^T}{\| X \|_F^2}  \right) \right) \left\| \frac{U - X}{\| X \|_F} \right\|_F^2
\end{align}

\bigskip

Since 
$$
\nabla^2 f(U) = \nabla^2 f(U; X) = \| X \|_F^2 \nabla^2 f \left( \frac{U}{\| X \|_F} ; \frac{X}{\| X \|_F} \right),
$$ 
we conclude that for general $X$, it holds for $\| U \|_F \leq \frac{3}{10 \| X \|_F}\lambda_r$ that 
\begin{align}
 \text{vec}\left( U - X \right)^T \nabla^2 f \left(U; X \right) \text{vec}\left(U - X \right) &\geq \frac{\lambda_r^2}{18 \| X \|_F^2} \left\| U - X \right\|_F^2  \nonumber \\
 \text{vec}\left( U - X \right)^T \nabla^2 f \left(U; X \right) \text{vec}\left(U - X \right) &\leq C \left(n^2r^2\frac{\lambda_r^2}{\| X \|_F^2} + \lambda_r \right) \left\| U - X \right\|_F^2 
\end{align}

\end{proof}

\subsection{Proofs for \ref{sec:init}, Initialization and Convergence}\label{sec:init_proofs}
\begin{proof}[Proof of Lemma \ref{thm:init}] It suffices to prove the case $\|X\|_F^2=1$.  By Corollary \ref{cor:full_concentration} we have that with probability greater than $1-2e^{-\beta rn}-6/m^2$
$$
\left\|\frac{1}{m}\sum_{i=1}^my_ia_ia_i^T-Id -2XX^T\right\|_{op} < \delta
$$
so long as $m\geq C\delta^{-2}\beta nr\log(n)^2$.  This implies 
$$
\left\|M-(1/2)Id -XX^T\right\|_{op} < \delta/2.
$$
Let
$$
A := M-(1/2)Id
$$
and collect the unit normalized eigenvectors corresponding to the dominant $r$-dimensional subspace of $A$ in a matrix $U\in\mathbb{R}^{n\times r}$. Let $\sigma_1\geq \sigma_2...\geq \sigma_{r+1}>0$ denote the eigenvalues of the observed matrix $M$ and define
$$
\Sigma = \begin{bmatrix}\sigma_1 & 0 & ... & 0\\ 0 & \sigma_2 & ... & 0\\ ... & ... & 0 & \sigma_r \end{bmatrix}_{r\times r}-\sigma_{r+1}Id_{r\times r}.
$$
Let $U_0:=U\Sigma^{1/2}$, and $Q:=O_1O_2^T$ where $(X^TX)^{-1/2}X^TU = O_1DO_2^T$ is the singular value decomposition and observe
\begin{subequations}
\begin{align}
\left\|U\Sigma^{1/2}-XQ\right\|_F &= \left\|U\Sigma^{1/2}-U(X^TX)^{1/2}+U(X^TX)^{1/2}-XQ\right\|_F\\
&\leq \left\|U-X(X^TX)^{-1/2}Q\right\|_F\left\|Q^T(X^TX)^{1/2}\right\|_{op} + \left\|U\right\|_{op}\left\|\Sigma^{1/2}-(X^TX)^{1/2}\right\|_F\\
&\leq \frac{2^{3/2}\sqrt{r}\|A-XX^T\|_{op}}{\lambda_r}\sqrt{\lambda_1} + \left\|\Sigma^{1/2}-(X^TX)^{1/2}\right\|_F\label{eq:davis_kahan}\\
&\leq \frac{2^{1/2}\sqrt{r}\delta}{\lambda_r}\sqrt{\lambda_1} + \frac{1}{2\lambda_r}\left(\sqrt{r}\delta+ \delta\right)\label{eq:sigma_bound}
\end{align}
\end{subequations}
where \eqref{eq:davis_kahan} follows from Theorem 2 in \cite{yu2015useful} and \eqref{eq:sigma_bound} follows from
\begin{subequations}
\begin{align*}
\left\|\Sigma^{1/2}-(X^TX)^{1/2}\right\|_F &= \sqrt{\sum_{i=1}^r\left|\sqrt{\sigma_i-\sigma_{r+1}}-\sqrt{\lambda_i}\right|^2}\\
&= \sqrt{\sum_{i=1}^r\left|\frac{\sigma_i-\sigma_{r+1}-\lambda_i}{\sqrt{\sigma_i-\sigma_{r+1}}+\sqrt{\lambda_i}}\right|^2}\\
&\leq \frac{1}{\sqrt{\lambda_r}}\left\|M-XX^T\right\|_F\\
&\leq  \frac{1}{\sqrt{\lambda_r}}\left(\left\|A-XX^T\right\|_F + \left|\sigma_{r+1}-1/2\right|\right)\\
&\leq  \frac{1}{2\lambda_r}\left(\sqrt{r}\delta+ \delta\right).
\end{align*}
\end{subequations}
The first equality holds because $X$ has orthogonal columns and thus $X^TX$ is a diagonal matrix. 

Now, we have
\begin{subequations}
\begin{align*}
\frac{2^{1/2}\sqrt{r}\delta}{\lambda_r}\sqrt{\lambda_1}&<\frac{3}{20}\lambda_r\\
\frac{1}{2\lambda_r}\left(\sqrt{r}\delta + \delta\right)&<\frac{3}{20}\lambda_r
\end{align*}
\end{subequations}
so long as
$$
\delta < C(\sqrt{r})^{-1}\min(\lambda_r^2/\sqrt{\lambda_1}, \lambda_r^2) = Cr^{-1/2}\lambda_r^2
$$
which gives the stated claim.
\end{proof}

\subsection{Proofs for Rank-one Matrix Recovery, \S\ref{sec:rank_one}}
\subsubsection{Proof of Lemma \ref{lem:exp_formula}}\label{sec:exp_rank1}
\begin{proof} Note that by \eqref{eq:hessian} we only need to compute $\mathbb{E}[(a^Tu)^2aa^T]$ for an arbitrary $u\in\mathbb{R}^n$.  We will consider the slightly more general expectation $\mathbb{E}[(a^Tu)(a^Tw)aa^T]$ for arbitrary $u,w\in\mathbb{R}^n$.  Begin by assuming $\Sigma = Id$ where $Id$ is the $n\times n$ identity matrix.    Let $i,j\in[n]$ be arbitrary coordinates.  We have
\begin{equation*}
\begin{split}
\mathbb{E}[(a^Tu)(a^Tw)a_i^2] &= \mathbb{E}[a_i^2\sum_{k=1}^na_k^2u_kw_k + a_i^2\sum_{k\neq j}a_ka_ju_kw_j]\\
&=\mu_4u_iw_i + \sum_{k\neq i}u_kw_k\\
&=(\mu_4-1)u_iw_i + u^Tw
\end{split}
\end{equation*}
and for $i\neq j$, 
\begin{equation*}
\begin{split}
\mathbb{E}[(a^Tu)(a^Tw)a_ia_j] &= \mathbb{E}[a_ia_j\sum_{k=1}^na_k^2u_kw_k + a_ia_j\sum_{k\neq l}a_ka_lu_kw_l]\\
&=u_iw_j + w_ju_i
\end{split}
\end{equation*}
so that
\begin{equation}\label{eq:main_expectation}
\mathbb{E}[(a^Tu)(a^Tw)aa^T] = (u^Tw)Id + uw^T +wu^T+ (\mu_4-3)\sum_{k=1}^nu_kw_ke_ke_k^T.
\end{equation}
If $\Sigma\neq Id$, then observe that if we define $b:=\Sigma^{-1/2}a$,
$$
(a^Tu)^2aa^T = \Sigma^{1/2}(b^T\Sigma^{1/2}u)^2bb^T\Sigma^{1/2}
$$
then the inner term satisfies the assumptions needs for \eqref{eq:main_expectation}, and so we find
\begin{equation*}
\begin{split}
\mathbb{E}[(a^Tu)^2aa^T] &= \Sigma^{1/2}\left(\|\Sigma^{1/2}u\|_2^2Id + 2\Sigma^{1/2}uu^T\Sigma^{1/2} + (\mu_4-3)\sum_{k=1}^n(\Sigma^{1/2}u)_k^2e_ke_k^T\right)\Sigma^{1/2}\\
&=\|\Sigma^{1/2}u\|_2^2\Sigma + 2\Sigma uu^T\Sigma + (\mu_4-3)\sum_{k=1}^n(v_k^Tu)^2v_kv_k^T
\end{split}
\end{equation*}
where $\Sigma^{1/2} = \begin{bmatrix}v_1&v_2&...&v_n\end{bmatrix}_{n\times n}$.  
\end{proof}

\subsubsection{Proof of Lemma \ref{lem:quant_lb}}\label{sec:quant_lb_proof}
We begin with an eigenvalue bound.
\begin{lem} \label{lemma:one_half} Let $u\in\mathbb{R}^n$ be a unit vector and consider the traceless matrix $$Z:=uu^T - \sum_{k=1}^nu_k^2e_ke_k^T.$$ Suppose that 
$$
0\leq u_n^2\leq u_{n-1}^2\leq...\leq u_1^2.
$$ Then we have
\begin{equation*}
\begin{split}
\lambda_{min}(Z)&\in [-\min\{u_1^2,1/2\},-u_2^2]\\
\lambda_{max}(Z)&\in [0,1-u_n^2).
\end{split}
\end{equation*}
\end{lem}

\begin{proof}
Suppose first that $u_n^2>0$. Then we can use the determinant formula
\begin{equation}\label{eq:determinant}
\begin{split}
\det\left(Z-\lambda Id\right) &= \det\left(\sum_{k=1}^n(-u_k^2-\lambda)e_ke_k^T\right)\cdot\left(1-u^T\left(\sum_{k=1}^n(u_k^2+\lambda)^{-1}e_ke_k^T\right)u\right)\\
&=\prod_{k=1}^n(-u_k^2-\lambda)\cdot\left(1-\sum_{k=1}^n\frac{u_k^2}{u_k^2+\lambda}\right)
\end{split}
\end{equation}
which shows us the following:
\begin{enumerate}
\item If \emph{any} $u_j^2=u_{j+1}^2$ then $\lambda = -u_j^2$ is an eigenvalue.
\item If all of the squared coordinates are distinct then each eigenvalue $\lambda$ satisfies
$$
\sum_{k=1}^n \frac{u_k^2}{u_k^2+\lambda} = 1
$$
and because there will be a vertical asymptote at each $-u_k^2$ we see the eigenvalues of $Z$ interlace the squared coordinates, the smallest occurring somewhere between $(-u_1^2,-u_2^2)$ and the largest somewhere after $-u^2_n$.
\end{enumerate}
In general, if some of the coordinates are $0$, we see that with the assumed ordering $Z$ will be a block matrix and we can apply \eqref{eq:determinant} to the reduced space where $Z$ acts nontrivially.

The bound $\lambda_{min}\geq -1/2$ follows from
$$
\min_{\|y\|_2=1=\|z\|_2}\left(\sum_{k=1}^ny_kz_k\right)^2 - \sum_{k=1}^ny_k^2z_k^2 \geq -1/2.
$$
Note that by the Gershgorin Circle Theorem, the largest eigenvalue $\lambda_{max}$ is no larger than $1-u_n^2$.  
\end{proof}

\begin{proof}[Proof of Lemma \ref{lem:quant_lb}] Suppose that $\|x\|_2=1$ and let
$$
g(\mu):=\lambda_{min}\left(Id +2xx^T+(\mu-3)\sum_{k=1}^nx_k^2e_ke_k^T\right).
$$
Lemma \ref{lemma:one_half} above shows $g(1)\geq 1-\min\{2\tau(x),1\}$ and it is clear that $g(3)=1$.  By concavity of $\lambda_{min}(\cdot)$ we then find
\begin{subequations}
\begin{align*}
g(\mu)&\geq \min\{\tau(x),1/2\}\mu + 1-3\min\{\tau(x),1/2\}\\
&=1+\min\{\mu-3,0\}\min\{\tau(x),1/2\}
\end{align*}
\end{subequations}
for all $\mu\in[1,3]$.  

If $\mu_4>3$, then $2xx^T+(\mu_4-3)\sum_{k=1}^nx_k^2e_ke_k^T$ is a positive semi-definite matrix and thus $g(\mu)\geq 1$.  Observing that
$$
2g(\mu)\|x\|_2^2 = \lambda_{min}\left(\mathbb{E}\left[\nabla^2 f(x)\right]\right)
$$
produces the bound \eqref{eq:lb}.
\end{proof}

\subsubsection{Proof of Lemma \ref{lem:exp_convexity1}}
\begin{proof}Begin by assuming $\Sigma=Id$ and reparametrize an arbitrary $u\in\mathbb{R}^n$ as $u=x-tw$ for $\|w\|_2=1$.  Note that
\begin{equation}\label{eq:hessian1}
\nabla^2f(x-tw) = \frac{1}{m}\sum_{i=1}^m\left(2(a_i^Tx)^2-6(a_i^Tx)(a_i^Tw)t+3(a_i^Tw)^2t^2\right)a_ia_i^T
\end{equation}
so using \eqref{eq:main_expectation} we find
\begin{equation*}\label{eq:hessian2}
\begin{split}
\mathbb{E}[\nabla^2f(x-tw)] &= 3\left[Id +2ww^T+(\mu_4-3)\sum_{k=1}^nw_k^2e_ke_k^T\right]t^2\\
&\ \ \ \ -6\left[(x^Tw)Id + xw^T +wx^T+ (\mu_4-3)\sum_{k=1}^nx_kw_ke_ke_k^T\right]t \\
&\ \ \ \ + 2\left[\|x\|_2^2Id +2xx^T+(\mu_4-3)\sum_{k=1}^nx_k^2e_ke_k^T\right].
\end{split}
\end{equation*}
Observe that
\begin{equation}\label{eq:constant_term}
\begin{split}
\|x\|_2^2Id +2xx^T+(\mu_4-3)\sum_{k=1}^nx_k^2e_ke_k^T &\succeq \left(1+\min\{\tau(x),1/2\}[\mu_4-3]_-\right)\|x\|_2^2Id
\end{split}
\end{equation}
and that
\begin{equation}\label{eq:linear_term}
(x^Tw)Id + xw^T +wx^T+ (\mu_4-3)\sum_{k=1}^nx_kw_ke_ke_k^T \preceq \left(3+\tau(x)[\mu_4-3]_+\right)\|x\|_2Id.
\end{equation}
For \eqref{eq:constant_term}, we used Lemma \ref{lem:quant_lb}.  Lastly, note that 
$$
Id +2ww^T+(\mu_4-3)\sum_{k=1}^nw_k^2e_ke_k^T \succeq 0.
$$

Consequently we can define the polynomials
$$
Q_{y,w}(t) := y^T\mathbb{E}[\nabla^2f(x-tw)]y
$$
and by convexity we can bound the smallest positive root by the intercept of the tangent line;  the bounds \eqref{eq:constant_term} and \eqref{eq:linear_term} thus yield the stated conclusion for $\Sigma=Id$.

For general covariance matrices, note that we have just shown that
$$
\Sigma^{-1/2}\mathbb{E}[\nabla^2f(u)]\Sigma^{-1/2}\succeq 0
$$
whenever $\Sigma^{1/2}u$ and  $\Sigma^{1/2}x$ are close enough, which implies
$$
\mathbb{E}[\nabla^2f(u)]\succeq 0.
$$
\end{proof}

\subsubsection{Proof of Lemma \ref{lem:concentration1}} \label{sec:proof_concentration1}
\begin{proof}
Begin by assuming $\Sigma=Id$ and $\|x\|_2=1$.  Note that because of the sub-gaussian assumption we have that for $m\geq C$
\begin{equation*}\label{eq:tail_bounds}
\begin{split}
\mathbb{P}\left[(a_i^Tx)^2\geq c\log m\right]&\leq \exp\left(1-\hat{c}\log m\right)\\
&\leq m^{-4}\\
\mathbb{P}\left[\|a_i\|_2^2\geq cn\log m\right]&\leq 2\exp\left(-\hat{c}\log m\right)\\
&\leq m^{-4}
\end{split}
\end{equation*}
where the constants depend on the sub-gaussian norm of $a_i$.  Consequently 
$$
\mathbb{P}\left[\max_{1\leq i\leq m}(a_i^Tx)^2\|a_i\|_2^2\geq c^2n(\log m)^2\right]\leq 2m(m^{-4}) \leq 2 m^{-3}
$$
and if we define the truncated random variables $\tilde{a}_i:=a_i\chi_{(a_i^Tx)^2\|a_i\|_2^2\leq c^2n(\log m)^2}$ and the analogous truncated matrix 
$$\tilde{M} = \frac{1}{m} \sum_{i=1}^m (\tilde{a}_i^{T} x)^2 \tilde{a}_i \tilde{a}_i^{T},$$
Bernstein's inequality (Theorem 4.1 in \cite{tropp2012user}) tells us that
$$
\mathbb{P}\left[\left\|\tilde{M}-\mathbb{E}[\tilde{M}]\right\|_{op}\geq \delta\right]\leq n\exp\left(\frac{-\delta^2m^2/2}{\sigma^2+mn(\log m)^2\delta/3}\right)
$$
where $\sigma^2$ is given by
\begin{equation}\label{eq:sigma_squared}
\begin{split}
\sigma^2&:= \left\|\sum_{i=1}^m\mathbb{E}\left[(\tilde{a}_i^Tx)^4\|\tilde{a}_i\|_2^2\tilde{a}_i\tilde{a}_i^T\right]-\left(\mathbb{E}\left[(\tilde{a}_i^Tx)^2\tilde{a}_i\tilde{a}_i^T\right]\right)^2\right\|_{op}\\
&=m\left\|\mathbb{E}\left[(\tilde{a}^Tx)^4\|\tilde{a}\|_2^2\tilde{a}\tilde{a}^T\right]-\left(\mathbb{E}\left[(\tilde{a}^Tx)^2\tilde{a}\tilde{a}^T\right]\right)^2\right\|_{op}\\
&\leq Cmn
\end{split}
\end{equation}
where $C$ is a constant which depends on the moments of $a_i$.

Lastly observe that  if $\hat{a}_i:=a_i\chi_{(a_i^Tx)^2\|a_i\|_2^2\geq c^2n(\log m)^2}$ then we can write
\begin{equation}\label{eq:expectations_are_close}
\begin{split}
\left\|\mathbb{E}[\tilde{M}]-\mathbb{E}[M]\right\|_{op} &= \left\|\mathbb{E}\left[(\hat{a}^Tx)^2\hat{a}\hat{a}^T\right]\right\|_{op}\\
&\leq \mathbb{E}\left[(\hat{a}^Tx)^2\|\hat{a}\|_2^2\right]\\
&=\int_0^{c^2n(\log m)^2}\mathbb{P}\left[(a^Tx)^2\|a\|_2^2\geq c^2n(\log m)^2\right]\ dt + \int_{c^2n(\log m)^2}^\infty\mathbb{P}\left[(a^Tx)^2\|a\|_2^2\geq t\right]\ dt\\
&=c^2n(\log m)^2\left(\mathbb{P}\left[(a^Tx)^2\|a\|_2^2\geq c^2n(\log m)^2\right]+\int_1^\infty\mathbb{P}\left[(a^Tx)^2\|a\|_2^2\geq \alpha c^2n(\log m)^2\right]\ d\alpha\right)\\
&\leq 3/m^2
\end{split}
\end{equation}
where we used Jensen's inequality for the first line and have assumed $m\geq Cn$.  Consequently we find that for $m\geq Cn$
\begin{equation*}\label{eq:master:eq}
\begin{split}
\mathbb{P}\left[\left\|M-\mathbb{E}[M]\right\|_{op}\geq \epsilon\right]&\leq \mathbb{P}\left[M\neq \tilde{M}\right]+\mathbb{P}\left[\left\|\tilde{M}-\mathbb{E}[M]\right\|_{op}\geq \epsilon\right]\\
&\leq m(2m^{-4}) + n\exp\left(\frac{-\delta^2m^2/2}{\sigma^2+mn(\log m)^2\delta/3}\right)
\end{split}
\end{equation*}
where $\delta:=\epsilon-3/m^2$ from \eqref{eq:expectations_are_close}.  Now, all we have left is to show that the exponential can be made less than a power of $m$.  Using \eqref{eq:sigma_squared} we find that we need $m$ to satisfy
$$
m\geq n\left(\log n+2\log m\right)\cdot\left(C+(\log m)^2\delta\right)/\delta^2
$$
for which it suffices to require $m\geq C\epsilon^{-2}n(\log n)^3$ for some constant $C$ which only depends on the moments of $a_i$.

For the more general statement note that our previous work shows
\begin{equation}\label{eq:covariance}
\left\|\frac{1}{m}\sum_{i=1}^m(b_i^T\Sigma^{1/2}x)^2b_ib_i^T-\mathbb{E}\left[(b^T\Sigma^{1/2}x)^2bb^T\right]\right\|_{op}<\epsilon\|\Sigma\|_{op}^{-1}\|\Sigma^{1/2}x\|_2^2
\end{equation}
whenever $m\geq C\epsilon^{-2}\|\Sigma\|_{op}^2n(\log n)^3$.  Consequently,
\begin{equation*}\label{eq:covariance2}
\begin{split}
\left\|\frac{1}{m}\sum_{i=1}^m(a_i^Tx)^2a_ia_i^T-\mathbb{E}\left[(a^Tx)^2aa^T\right]\right\|_{op} &=\left\|\Sigma^{1/2}\left(\frac{1}{m}\sum_{i=1}^m(b_i^T\Sigma^{1/2}x)^2b_ib_i^T-\mathbb{E}\left[(b^T\Sigma^{1/2}x)^2bb^T\right]\right)\Sigma^{1/2}\right\|_{op}\\
&\leq \|\Sigma\|_{op}\left\|\frac{1}{m}\sum_{i=1}^m(b_i^T\Sigma^{1/2}x)^2b_ib_i^T-\mathbb{E}\left[(b^T\Sigma^{1/2}x)^2bb^T\right]\right\|_{op}\\
&<\epsilon\|\Sigma^{1/2}x\|_2^2
\end{split}
\end{equation*}
which is the desired claim.
\end{proof}

\subsubsection{Proof of Theorem \ref{thm:main_convexity1}}\label{sec:proof_convexity1}
\begin{proof}[Proof of Theorem \ref{thm:main_convexity1}] Assume without loss that $\|x\|_2=1$ and that $\hat{w}$ is the normalized direction from $u$ to $x$. Moreover begin by assuming $\Sigma = Id$.  Closely following the proof of Theorem \ref{thm:main_convexity} we first note that
$$
f''(t) = \frac{1}{m}\sum_{i=1}^m 3(a_i^T\hat{w})^4t^2 + 6(a_i^T\hat{w})^3(a_i^Tx)t + 2(a_i^Tx\hat{w}^Ta_i)^2
$$
and define
$$
Z_i := \left(A_it +B_i\right)^2\geq 0
$$
where
\begin{subequations}
\begin{align*}
A_i&:=\left(a_i^T\hat{w}\hat{w}^Ta_i\right)\\
B_i&:=\left(a_i^Tx\hat{w}^Ta_i\right).
\end{align*}
\end{subequations}
Note that by the computations done in Lemma \ref{lem:exp_formula} we have
\begin{equation*}\label{eq:rank_one_mu}
\begin{split}
\mu(t)&:=\mathbb{E}[Z_i(t)]\\
&=\left(3+(\mu_4-3)\|\hat{w}\|_4^4\right)t^2 + 2\left(3x^T\hat{w}+(\mu_4-3)\sum_k x_k\hat{w}_k^3\right)t + \frac{1}{2}\hat{w}^T\mathbb{E}[\nabla^2f(x)]\hat{w} \\
&\leq \left(3+[\mu_4-3]_+\right)t^2 + \left(6+2[\mu_4-3]_+\right)t + \frac{3+2[\mu_4-3]_+}{2}
\end{split}
\end{equation*}
and that the variance
$$
\mathbb{E}[\left(Z_i(t)-\mu(t)\right)^2] \leq C^2(t)
$$
where $C(t)$ depends only on the subgaussian norm of the $a_i$. 

Now, for a given $\epsilon\in(0,1)$ define
\begin{subequations}
\begin{align*}
Y_i(t) &:= \mu(t)-Z_i(t)\\
b(t) &:= \left(3+[\mu_4-3]_+\right)t^2 + \left(6+2[\mu_4-3]_+\right)t + \frac{3+2[\mu_4-3]_+}{2}\\
v^2(t) &:= C(t)^2\\
\sigma^2(t) &:= mC(t)^2\\
y &:= m\lambda_{min}(\mathbb{E}[\nabla^2f(x)])/12 = m\lambda/12.
\end{align*}
\end{subequations}
Applying Lemma \ref{lem:candes} above yields
$$
\mathbb{P}\left[\mu(t)-\frac{1}{m}\sum_{i=1}^mZ_i(t) \geq \lambda/12\right]\leq \min\left\{\exp\left(-\frac{\lambda^2m}{144C(t)^2}\right),25\left(1-\Phi(\frac{\lambda\sqrt{m}}{12C(t)})\right)\right\}.
$$
Using the well-known bound 
$$
1-\Phi(\frac{\lambda\sqrt{m}}{12C(t)}) < \frac{12C(t)}{\lambda\sqrt{2m\pi}}\exp\left(-\frac{\lambda^2m}{288C(t)^2}\right)
$$
we find that if $m\geq 288\alpha\lambda^{-2}C(t)^2n$ then with probability at least $1-e^{-\alpha n}$ we have
\begin{equation}\label{eq:e_net_lb_one}
\begin{split}
f''(t) &= \frac{3}{m}\sum_{i=1}^m\left(A_it+B_i\right)^2 - \frac{1}{m}\sum_{i=1}^mB_i^2\\
&\geq 3\mu(t) -\lambda/4 - \frac{1}{m}\sum_{i=1}^mB_i^2\\
&\geq(9+3[\mu_4-3]_-)t^2 + (6[\mu_4-3]_- -3)t + \hat{w}^T\mathbb{E}\left[\nabla^2f(x)\right]\hat{w} + \frac{1}{2}\hat{w}^T\mathbb{E}\left[\nabla^2f(x)\right]\hat{w}-\lambda/4 - \frac{1}{2}\hat{w}^T\nabla^2f(x)\hat{w}\\
&\geq(9+3[\mu_4-3]_-)t^2 + (6[\mu_4-3]_- - 3)t + \hat{w}^T\mathbb{E}\left[\nabla^2f(x)\right]\hat{w} -\lambda/4 - \epsilon/2\\
&\geq(9+3[\mu_4-3]_-)t^2 + (6[\mu_4-3]_- -3)t + \lambda/2
\end{split}
\end{equation}
where we used the concentration guaranteed by Lemma \ref{lem:concentration1} above with $\epsilon < \lambda/8$ and the fact that 
$$
\frac{1}{m}\sum_{i=1}^mB_i^2 = \frac{1}{2}\hat{w}^T\nabla^2f(x)\hat{w}.
$$
Consequently, using a tangent line bound for the smallest positive root we find that for all $$0\leq t < \frac{\lambda}{6-12[\mu_4-3]_-}$$ we have
$$
f''(t)\geq \lambda^2/12.
$$

Moreover, an $\epsilon$-net argument over all directions $\hat{w}$ shows that \eqref{eq:e_net_lb_one} holds for an arbitrary $\hat{w}$ with probability at least $1-e^{-\beta n}$.

For general covariance matrices, apply the previous argument to $\Sigma^{1/2}u$ and $\Sigma^{1/2}x$ with measurements $b_i=\Sigma^{-1/2}a_i$ as usual.

\end{proof}

 \end{document}